\documentclass{ws-rv9x6}
\usepackage{ws-rv-thm}   % comment this line when `amsthm / theorem / ntheorem` package is used
\usepackage{ws-rv-van} 
\usepackage{graphicx}
\usepackage{mathrsfs}
\usepackage{wasysym}
\usepackage{nicefrac}
\usepackage[format=hang,font=sf,labelfont=bf]{caption}
\usepackage{pifont}
\DeclareMathAlphabet{\zap}{OT1}{pzc}{m}{it}
\newcommand{\rad}{\mbox{\bf \sc{r}}}
\def\CC{\mathbb C}

\DeclareMathOperator{\End}{End}

\def\ZZ{{\mathbb Z}}
\def\RR{{\mathbb R}}
\def\CP{{\mathbb C \mathbb P}}
\def\Dir{\not \! \partial}
\def\dir{\not \negthickspace D}

\DeclareMathOperator{\hull}{Hull}

\DeclareMathOperator{\Vol}{Vol}
\DeclareMathOperator{\Hom}{Hom}
\DeclareMathOperator{\Ind}{Ind}
\DeclareMathOperator{\kod}{Kod}
\DeclareMathOperator{\spanner}{span}
\def\barroman#1{\sbox0{#1}\dimen0=\dimexpr\wd0+1pt\relax
  \makebox[\dimen0]{\rlap{\vrule width\dimen0 height 0.06ex depth 0.06ex}%
    \rlap{\vrule width\dimen0 height\dimexpr\ht0+0.03ex\relax 
            depth\dimexpr-\ht0+0.09ex\relax}%
    \kern.5pt#1\kern.5pt}}
\begin{document}

\chapter[On the Scalar Curvature of $4$-Manifolds]{On the Scalar Curvature of $4$-Manifolds}

\author[Claude LeBrun]{Claude LeBrun}

\address{Department of Mathematics\\
Stony Brook  University\\
Stony Brook, NY 11794--3651\\
claude@math.stonybrook.edu}

%\maketitle

\begin{abstract} Dimension four provides a  peculiarly idiosyncratic  setting for  the
interplay  between  scalar curvature and differential topology. Here we will  explain some of the
peculiarities  of the four-dimensional  realm  
via a careful    discussion of the Yamabe invariant  (or ``sigma constant''). In the process, we 
will also prove some new results, and point out  open   problems that continue to  represent 
 key   challenges   in the subject.  \end{abstract}

\bigskip

\body

Much of  modern Riemannian geometry concerns  the relationship  between 
 curvature  and 
differential  
topology. In keeping with the overall theme of this book, the present chapter  will  describe some of 
what this dialog has    taught  us about the {\sf scalar curvature}  of Riemannian manifolds. 
Our specific  objective   will be  to  explain how 
  the interplay between  scalar curvature and differential topology in dimension four 
hinges on phenomena that are  
quite  unlike anything seen  in other dimensions. 

Recall that  the {\em scalar curvature} $s: M\to \RR$ of a Riemannian $n$-manifold
 $(M,g)$ is by definition  the full trace  
 $${s}= {{\mathcal R}^{ij}}_{ij}=g^{jk}{{\mathcal R}^{i}}_{jik}$$
of the Riemann curvature tensor.  This has a nice synthetic interpretation  in terms of the volumes 
of the Riemannian distance balls $B_{\varepsilon} ({p})$ of small 
radius $\varepsilon$ about a point 
$p\in M$, because   $$\frac{{\mbox{vol}}_{g} (B_{\varepsilon} ({p}))}{v_n{\varepsilon}^n}= 
1-{\frac{{s}(p)}{6(n+2)}}{\varepsilon}^2 + O ({\varepsilon}^4), $$
where $$v_n=\frac{\pi^{n/2}}{(n/2)!} : = \frac{\pi^{n/2}}{\Gamma (\frac{n}{2}+1)}$$
 is the volume of the unit
ball in $\RR^n$. Thus, in a region where $s> 0$, a tiny ball
in the manifold 
has smaller  volume than it would in Euclidean space, while    in a region where  $s < 0$, tiny balls  instead have 
 larger volumes than they would in  Euclidean space. 
Still, because
the scalar curvature  does not similarly control the volume of  moderately {\sf  large} balls in dimension $n \geq 3$,
 this  synthetic interpretation   does not immediately allow one  to prove much of anything, and the search for a deeper synthetic 
 understanding of the scalar curvature therefore remains an active area of current research. Fortunately, however, the study of  
 geometric partial differential equations  provides an alternative source of   insights into  the global differential-topological meaning of 
 the scalar curvature, and this article will largely focus on results that have come  to light by dint of such means. 
  
One of the most surprising lessons these methods have taught us  is that  the global  behavior of the scalar curvature  in dimension 
four is wildly  different from what occurs in   other dimensions. 
As a simple  illustration, given a smooth compact manifold\footnote{In this article, the term  {\sf manifold} is  construed  in the restrictive sense, and so   should always be understood to mean a manifold {\em without boundary}.} $M^n$, 
 we  might ask how small we can make the $L^p$ norm of the scalar curvature, among   all possible  metrics
on $M$.  For most values of $p$, this is of course a silly question, because 
one can usually make the  $L^p$ norm of the scalar curvature tend to zero  by simply 
choosing a  suitable sequence   of constant rescalings $cg$ of  some  fixed   metric $g$; the 
 question is therefore only interesting when $p=n/2$, as this is 
  the unique value of $p$  for which  the  $L^p$-norm of the scalar curvature is actually scale-invariant. 
We are thus led to consider  the scale-invariant, non-negative Riemannian functional 
\begin{equation}
\label{funky}
\mathfrak{S} (g )  =   \int_M |s_g|^{n/2}d\mu_g  = \| s_g\|^{n/2}_{L^{n/2},g}
\end{equation}
along with   the associated 
real-valued diffeomorphism invariant  given by 
\begin{equation}
\label{dory}
\mathcal{I}_s(M^n)  =\inf_g \mathfrak{S} (g ) ,
\end{equation}
where the infimum is taken over all smooth Riemannian metrics $g$ on $M$.  
Of course,  the  $n=2$ case is highly  atypical, because the classical Gauss-Bonnet theorem 
tells us  that 
$\mathcal{I}_s(M^2)= 4\pi | \chi (M^2) |$, and that 
  the infimum is achieved by any metric $g$ whose Gauss curvature $K=s/2$ doesn't change  sign. 
We will therefore  restrict  our attention henceforth  to the setting of  $n\geq 3$, where the problem takes on an entirely different character.
In this range,  the functional $\mathfrak{S} (g )$   depends {differentiably} on $g$,  and the  critical points of  \eqref{funky}
 are 
precisely the {\em Einstein metrics}, which is to say   the  metrics of constant Ricci curvature, and  
the {\em scalar-flat metrics}, meaning the metrics with $s\equiv 0$. 
Although this  might make it seem tempting to look for Einstein metrics by trying to minimize the non-negative functional   $\mathfrak{S}$,  
 such hopes are   usually destined  to  be  disappointed. 
Indeed, in dimensions other than four, the differential-topological  invariant $\mathcal{I}_s$ defined by \eqref{dory} 
turns out to be {\sf trivial}, at least  for simply-connected manifolds: 

\begin{theorem} 
\label{tour}
Let $M$ be a smooth compact  manifold of  dimension 
$n\geq 3$. If $M$ is {\sf simply-connected}, and if $n\neq 4$, then $\mathcal{I}_s(M)=0$. 
\end{theorem}

\noindent 
By contrast, however, the invariant $\mathcal{I}_s$   is highly non-trivial in dimension $4$:

\begin{theorem} 
\label{option} 
There are sequences
 $\{ M_j\}$ of  smooth compact {\sf simply connected} $4$-manifolds
 with $\lim_{j\to \infty}\mathcal{I}_s(M_j)=+ \infty$. Moreover, the
 manifolds $M_j$ may be chosen so  that,  for each $j$,   the infimum
 \eqref{dory} 
 is  achieved by an Einstein metric $g_j$ on $M_j$. 
 \end{theorem}

While this pair of  results suffices to  show that    dimension four is 
{\sf sui generis} when it comes to  questions about  the scalar curvature, we will soon see that 
the invariant $\mathcal{I}_s(M)$ is merely  a pale shadow of the {\em Yamabe invariant} $\mathscr{Y}(M)$, 
which requires more work  to define, but which often encodes  significantly more geometrical information
about a manifold. After defining $\mathscr{Y}(M)$ in \S \ref{yammer} below, and  explaining how it related 
to $\mathcal{I}_s(M)$, we will immediately see why Theorem \ref{tour}  is a direct consequence of  
results due to Gromov, Lawson, Stolz, Petean, and Perelman. By contrast, as will be 
explained  in \S \ref{yneg}, Theorem \ref{option}  is proved by means of 
  Seiberg-Witten theory, 
 and fits into a larger story  about $4$-manifolds with negative Yamabe invariant. Dimension four   is also the arena  for  
  unusual   results about   Yamabe invariants in the zero and positive cases, and the rather different methods used to prove such results 
 are  explored and explained   in \S \ref{yzero} and  \S\ref{ypos}, respectively.  In the process, we will encounter 
 some of the unsolved mysteries  of  the subject. One can only hope that   drawing attention to these questions 
will stimulate new activity  along     this beautiful   interface between  Riemannian geometry and differential topology.

\section{Yamabe Constants and Yamabe Invariants}
\label{yammer} 

In this section, we describe the {\sf Yamabe invariant} $\mathscr{Y}(M) \in \RR$ of a smooth compact $n$-manifold $M$, $n\geq 3$, 
and discuss some of its basic properties. This differential-topological  invariant quantitatively refines the question of  whether 
a given manifold admits positive-scalar-curvature metrics, because 
\begin{equation}
\label{positivity}
\mathscr{Y}(M) > 0  \quad \Longleftrightarrow  \quad M \mbox{ admits metrics with } s>0.
\end{equation}
On the other hand, we will also see that the Yamabe invariant also refines  the invariant defined by  
\eqref{dory}, because
\begin{equation}
\label{consign}
\mathcal{I}_s (M) 
=
\begin{cases}
    \quad 0 & \text{if } \mathscr{Y}(M)\geq 0, \\
      |\mathscr{Y}(M)|^{n/2}& \text{if } \mathscr{Y}(M)\leq 0.
\end{cases}
\end{equation}

The Yamabe invariant is a natural outgrowth of the study of the {\em total scalar curvature}, by which we mean  
the 
 {integral} of the scalar curvature with respect to the Riemannian volume measure. In dimension $2$, where the scalar curvature coincides with twice the Gauss curvature, the total scalar curvature just  computes the Euler characteristic,  
because the classical Gauss-Bonnet theorem tells us that 
$$\int_{M^2} s_g ~d\mu_g = 4\pi \chi (M^2)$$
for any Riemannian metric $g$ on a compact surface $M^2$. By contrast, when $n\geq  3$, the total scalar curvature is  ridiculously far from
being a topological invariant, because it is not even {\sf scale}  invariant;  instead, under constant rescaling $g\rightsquigarrow cg$ of  the metric, 
the total scalar curvature rescales by a factor of $c^{(n-2)/2}$. The standard way of remedying this last pathology is to divide by 
an appropriate power of the volume, which then results in  the so-called {\em normalized Einstein-Hilbert action}
\begin{equation}
\label{nehf}
\mathscr{E} (M^n,  g ) = \frac{\int_M s_g~d\mu_g}{(\int_M d\mu_g)^{1-\frac{2}{n}}}.
\end{equation}
However, when $n \geq 3$, this action $\mathscr{E}$  still depends quite sensitively on the metric. Indeed, 
it turns out to   neither be bounded above nor below, and its critical points 
turn out to  exactly be the  {\em Einstein metrics}  \cite[Theorem 4.21]{bes}.  

Although $\mathscr{E}$ has no lower bound  when $n \geq 3$, 
Yamabe nonetheless  discovered that its {\sf restriction} to any {\sf conformal class}
of metrics {\sf is}  always bounded below. To see this, we  set  $p= \frac{2n}{(n-2)}> 2$, which  is the unique real number so that 
a metric $\widehat{g}$ conformally related to $g$ and its  corresponding metric volume measure can  
be simultaneously expressed as 
$$\widehat{g}= u^{p-2}g \quad \mbox{and} \quad \widehat{d\mu} = u^p \,d\mu.$$
 With this choice, 
Yamabe then discovered that 
the scalar curvature $\widehat{s}$ of the rescaled metric $\widehat{g}$  is given by the remarkably simple equation 
$$\widehat{s} u^{p-1} = \left[(p+2) \Delta + s\right] u,$$
where $\Delta = d^*d = -\nabla\cdot  \nabla$. Given this, it then follows that 
\begin{equation}
\label{restricted}
\mathscr{E} (\widehat{g}) = \frac{\int_m\left[ (p+2) |\nabla u|^2 + s u^2 \right]\, d\mu}{\| u\|_{L^p}^2}
\end{equation}
and, by focusing on  functions of the form $u=1+tf$, where $t\in \RR$ and $\int_M f~d\mu =0$,  
one  therefore easily  sees that $g$ is a critical point  of  $\mathscr{E}|_{[g]}$ iff its scalar curvature 
${s}$ is constant; and since we could have just as easily chosen any other metric
in the conformal class as our reference metric,  it  now immediately  follows  that $\widehat{g}\in [g ]:= \{ u^{p-2}g \}$ 
is a critical point of $\mathscr{E}|_{[g]}$ iff $\widehat{s}$ is constant. 
In light of this, Yamabe therefore  set out to prove that each conformal class $\gamma= [g]$ contains a metric 
of constant scalar curvature by showing that minimizers of $\mathscr{E}|_\gamma$ always exist.  This claim   is actually 
 correct, and we therefore  now call 
 such  minimizers    {\sf Yamabe metrics}. Moreover,  Yamabe's  strategy 
 for proving their existence  was largely on target. Namely, he proved the existence of  a unit-norm 
 minimizing functions $u_\epsilon > 0$  for the modifications of  \eqref{restricted} obtained  by  replacing  $p$ with $p-\epsilon$, and  then asserted  that  these 
 $u_\epsilon$ converge  to a smooth positive function $u$ 
 as $\epsilon\searrow 0$.  However, while Yamabe's posthumously-published proof \cite{yamabe} 
 hinged   on the claim that the $u_\epsilon$ are uniformly bounded in $C^0$, 
 Neil Trudinger \cite{trud}  observed   that this 
  fails for metrics  in the  conformal class of   the usual    sectional-curvature $(+1)$ metric $g_0$ on $S^n$. Fortunately, 
 Trudinger  went on to  show  that Yamabe's outline does indeed work 
whenever  the {\sf Yamabe constant}
 $$
 Y(M, [g]) := \inf_{\widehat{g}\in [g]} \mathscr{E} (\widehat{g})
 $$
is non-positive, and then laid out a general approach to the  more difficult
 positive case, based  on  trying to show that the $u_\epsilon$  instead converge in a suitable Sobolev space. 
 By carefully analyzing the differential-geometric meaning of the best Sobolev  constant for $L^2_1\hookrightarrow L^p$, 
 Aubin \cite{aubyam}  then discovered  that Trudinger's method  actually works whenever  
 $$
  Y(M^n, [g]) < \mathscr{E}( S^n, g_0),
 $$
while  also observing  (see Figure \ref{balloon})   that one always at least has 
 \begin{equation}
\label{aubineq}
  Y(M^n, [g]) \leq \mathscr{E}( S^n, g_0)
\end{equation}
  for any compact Riemannian manifold $(M^n, g)$, $n\geq 3$. Having thereby reduced Yamabe's problem to that of showing 
  the equality case in \eqref{aubineq} only occurs when $(M^n, [g])$ is the standard $n$-sphere, Aubin  went on to prove that this 
is automatically 
  true  except    perhaps when $n\leq 5$  or  $[g]$ is locally conformally flat.  Schoen \cite{rick} then completed the proof of Yamabe's  claim 
  by using the Schoen-Yau positive 
  mass theorem    to eliminate  all the remaining cases.

 \begin{figure}[htb]
 \centerline{
 \includegraphics[scale=.9]{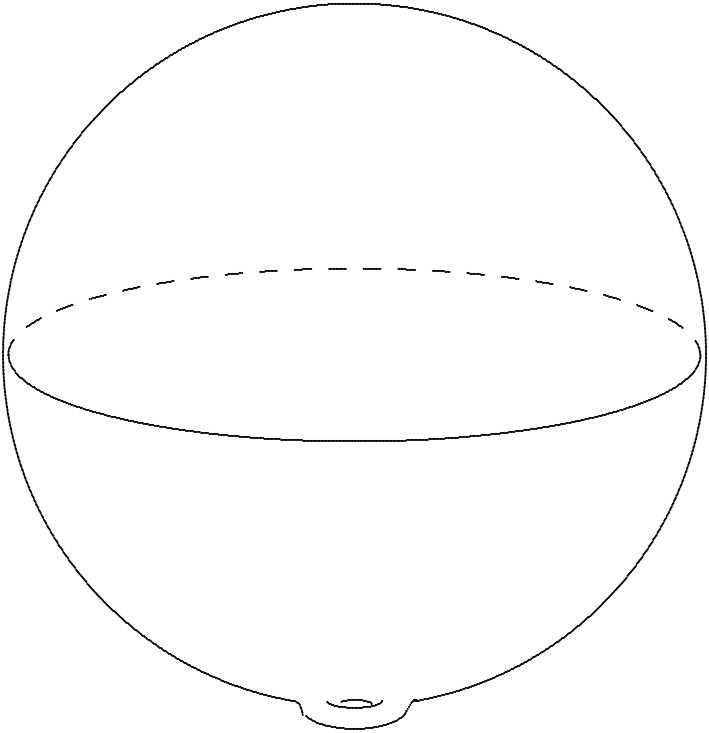}
 }
\caption{\label{balloon}
Because any Riemannian metric is nearly Euclidean on a sufficiently tiny scale, any compact Riemannian manifold has conformal rescalings that resemble a large, nearly spherical balloon with a tiny, topology-laden gondola attached. In any dimension $n\geq 3$, this means  that, for any $g$ and any $\epsilon > 0$, one can always construct conformal rescalings $\widehat{g} = u^{p-2}g$ with $\mathscr{E} (M^n, \widehat{g}) < \mathscr{E} (S^n, g_0)+ \epsilon$.}
\end{figure}

An oral tradition alleges that Yamabe's ultimate goal was to construct Einstein metrics by next {\sf maximizing}  $Y(M^n, \gamma)$ over the set of conformal classes
$\gamma$, as illustrated in Figure \ref{saddle}. While this program is now definitely known to  fail on many low-dimensional manifolds, Yamabe's dream  does at least 
give rise to a fascinating differential-topological invariant. Indeed,  the {\sf Yamabe invariant} of any smooth compact $n$-manifold $M$, $n\geq 3$, is defined as 
\begin{equation}
\label{yamabdef}
\mathscr{Y}(M) = \sup_\gamma  Y(M, \gamma) = \sup_\gamma \inf_{g\in \gamma} \mathscr{E}( M, g),
\end{equation}
where the supremum is taken over all conformal classes $\gamma = [g]$ of smooth Riemannian metrics on $M$,
and so automatically  satisfies
\begin{equation}
\label{ceiling}
  \mathscr{Y}(M^n) \leq \mathscr{E}( S^n, g_0) = \mathscr{Y}(S^n) 
\end{equation}
by \eqref{aubineq}.
This invariant was apparently   introduced independently  by Schoen \cite{sch} and Kobayashi \cite{okob}, 
who respectively called $\mathscr{Y}(M)$ the {\em sigma constant} 
and the {\em mu invariant} of $M$. While  both of these alternative  terminologies continue to have  proponents, I personally feel   that 
it is usually preferable to name  objects after mathematicians  rather than  after commonly-used Greek letters. 
  Nonetheless, the terminology adopted here still requires a bit of disambiguation, because one must be 
 careful not to confuse   Yamabe {\sf invariants} (of smooth  manifolds) with Yamabe {\sf constants} (of  specific conformal classes).

%%%saddle:

 \begin{figure}[htb]
  \centerline{
 \includegraphics[scale=.9]{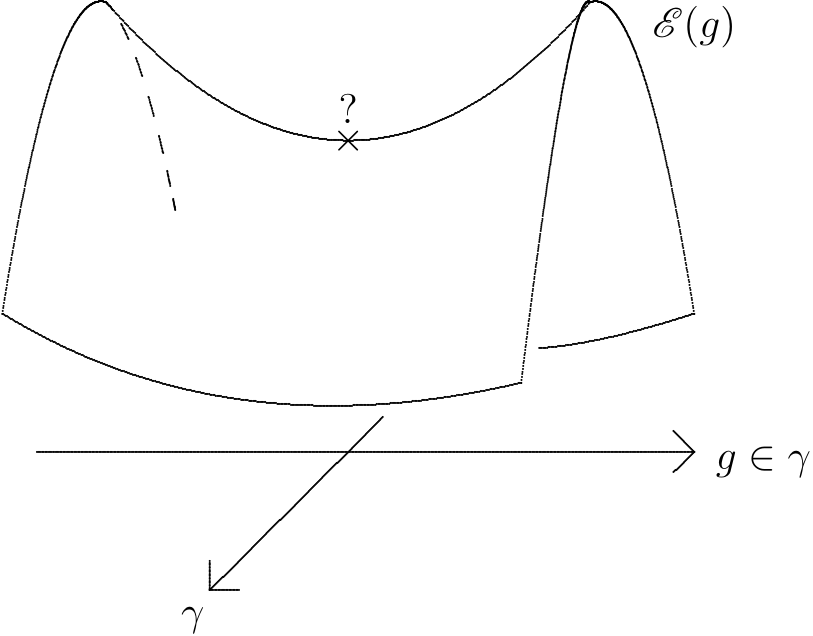}
 }
\caption{\label{saddle}
Yamabe apparently dreamt of finding  Einstein metrics on  compact $n$-manifolds, $n \geq 3$, by minimizing $\mathscr{E}$ over
each conformal class, and then maximizing over conformal classes. While a direct implementation of this   scheme is usually  
 destined to fail,  this idea nonetheless allows us to attach an important  real number, called the {\sl Yamabe invariant}, to every smooth 
compact manifold.}
\end{figure}

While the normalized Einstein-Hilbert functional $\mathscr{E}$ is  technically simpler than  the functional $\mathfrak{S}$ 
of \eqref{funky},  the two are actually closely related. Indeed, applying the H\"older inequality to \eqref{nehf} immediately yields  
\begin{equation}
\label{uptick}
 [\mathfrak{S}(g)]^{2/n}\geq \mathscr{E}(g)
\end{equation}
for any metric $g$, with equality iff $s_g = \mbox{const} \geq 0$. On the other hand, if $g$ and $\widehat{g}=u^{p-2}g$ 
are any two conformally related metrics, applying the same H\"older inequality to \eqref{restricted}, after first suppressing  the $|\nabla u|^2$ term, 
implies that
\begin{equation}
\label{downtick}
\mathscr{E}(\widehat{g}) \geq - [\mathfrak{S}({g})]^{2/n}
\end{equation}
with equality iff $u = \mbox{const}$ and $s_g=  \mbox{const} \leq 0$. Combining  \eqref{uptick} and \eqref{downtick},
we therefore conclude that 
any 
Yamabe metric $g$ can alternatively  be characterized as a minimizer of $\mathfrak{S}$  in its 
conformal class $\gamma = [g]$.

Notice that the equality case of \eqref{downtick}  not only  implies that any 
 metric $g$ with $s_g=  \mbox{const} \leq 0$ is  a Yamabe metric, 
but  moreover shows that, up to constant rescaling, such a $g$  
is actually the {\sf unique} Yamabe metric in its conformal class. By contrast, however, a metric $g$ with 
$s_g=  \mbox{const} > 0$ may often {\sf not} be a Yamabe metric, and this  nuisance  can in practice make it 
quite difficult
to calculate the exact value of $\mathscr{Y}(M)$ in the positive case. One of the few  tools available to help 
address this issue    is a result of Obata  \cite{oba} which 
implies  that any {\sf Einstein} metric $g$ is  a Yamabe metric. Indeed, Obata's theorem  says that, 
except in the case of $(S^n, [g_0])$, 
any Einstein metric $g$  is, modulo constant rescalings,  the {\sf only} constant-scalar-curvature metric, 
and hence  the {\sf unique} Yamabe metric,  in  its conformal class.

Let us now prove  assertions \eqref{positivity} and \eqref{consign}. 
To show   \eqref{positivity}, first notice that 
the functional \eqref{restricted} is manifestly positive on the conformal class $[g] = \{ u^{p-2} g\}$  of any metric $g$ with $s_g > 0$.
The existence of Yamabe minimizers
%\footnote{This, of course, is overkill, because  one could  just as well use one of the minimizers
%$u_\epsilon$ constructed by Yamabe for this purpose. Indeed, 
%as observed by Trudinger \cite{trud}, it is quite sufficent to simply minimize the Rayleigh quotient  obtained from \eqref{restricted} by replacing $p$ with $2$. Thus, $Y(M,[g])$ always has the same sign as
%the smallest eigenvalue $\lambda_{Y}(g)$ of the Yamabe Laplacian $(p+2) \Delta + s$.} 
therefore tells us, in particular, that   $Y(M, [g])> 0$ iff and only if $[g]$ contains a metric 
$\widehat{g}$ with $s_{\widehat{g}}> 0$. It thus follows that   $\mathscr{Y}(M) > 0$ 
iff $M$ carries a metric $g$ of positive scalar curvature, which is exactly 
statement \eqref{positivity}. 

To prove \eqref{consign}, it is helpful to first recall  the standard fact that $\mathscr{E}$ is  unbounded below
on any compact $n$-manifold $M$, $n\geq 3$.  A particularly simple proof of this fact goes as follows. 
 Let  $f: \RR^n\to \RR$ be  any smooth function that is supported in the ball of radius $\nicefrac{1}{2}$,  and then observe that 
the metric
$$g_f := e^f (dx^1)^2+ e^{-f} (dx^2)^2 + (dx^3)^2+\cdots + (dx^n)^2$$
has the same volume form as the Euclidean metric, but 
satisfies\footnote{For example, one can prove  \eqref{shazzam} by using
   O'Neill's formulas   to show that    $s_{g_f}$ is everywhere equal to   $-\frac{1}{2}|df^\perp|^2$ plus  the  scalar curvature of the relevant fiber of the Riemannian submersion 
    $(x^1, x^2 , x^3, \cdots , x^n) \mapsto (x^3, \cdots , x^n)$. Indeed, \cite[equation (9.37)]{bes}  simplifies dramatically in our case, because the fibers of the submersion are minimal and orthogonal to a flat transverse foliation. Using   classical Gauss-Bonnet to integrate out the  scalar curvature of the fibers then proves \eqref{shazzam}. 
    
 Two other proofs involve viewing the restriction of   $g_f$ to $[-\frac{1}{2}, \frac{1}{2}]^n$
           as defining  a smooth metric $\breve{g}_f$ on the $n$-torus
    $\RR^n/\ZZ^n$. For example, after multiplying by  $\RR/\ZZ$ if necessary, one obtains an almost-K\"ahler 
    metric on  an even-dimensional torus $\mathbb{T}^{2m}$ that is adapted to the standard symplectic form 
    $\omega = dx^1 \wedge dx^2 +\cdots +dx^{2m-1}\wedge dx^{2m}$, and 
     Blair's formula \cite{blair} for the total Hermitian scalar curvature
    then yields  \cite[Appendix]{lebdelpezzo} 
    a conceptual    proof of \eqref{shazzam}.  Alternatively, one may instead  apply the second variation   formula  
    \eqref{jimmy} inductively to the minimal hypersurfaces $\mathbb{T}^m\subset \mathbb{T}^{m+1}$, $m=2, \ldots , n-1$, for the variation $u\equiv 1$ with $\mathscr{A}^{\prime\prime}=0$,
 while integrating out the auxiliary variables at each stage of the induction.} 
\begin{equation}
\label{shazzam} 
\int_{\RR^n} s_{g_f}d\mu_{g_f} =-\frac{1}{2} \int_{\RR^n} |d^\perp f|^2 d\mu_{g_f},
\end{equation}
where $d^\perp f := \sum_{j=3}^n \frac{\partial f}{\partial x^j}dx^j$. 
Thus, if we start with an otherwise arbitrary metric $h$ on $M^n$ that contains a Euclidean unit ball, and then replace $h$ 
on  this  ball by $g_f$,  the total volume will remain unchanged, while  the total scalar curvature will be reduced by 
an arbitrarily  large amount 
provided  we take  $f$ to be sufficiently  oscillatory. Since $\mathscr{E}(g) \geq Y(M,[g])$ for any metric $g$, 
this 
immediately implies that any given  $M^n$, $n\geq 3$,  admits sequences of conformal classes $[g_j]$ with $Y(M, [g_j])\to -\infty$.

If  $\mathscr{Y}(M) > 0$, it therefore follows that $M$ admits  a pair of conformal classes $[h_+]$ and $[h_-]$ such that 
$Y(M, [h_+]) > 0$ and $Y(M, [h_-]) < 0$. On the other hand it is not difficult to show,   by direct inspection of \eqref{restricted}, 
 that $Y(M, [g])$ is a continuous function of
$g$ in the $C^2$ topology  \cite{bes}. Since the space of metrics is  connected, the intermediate value theorem 
 therefore implies  the existence of  conformal classes  $[h_0]$ on $M$ with $Y(M, [h_0]) = 0$, and hence 
of Yamabe  metrics $\widehat{h}_0$ on $M$ with $s\equiv 0$. Hence  $\mathcal{I}_s(M)=0$ whenever  $\mathscr{Y}(M)> 0$. On the other hand, 
   \eqref{downtick} implies that $\inf_{g\in \gamma} \mathfrak{S}(g) = [-Y(M,\gamma)]^{n/2}$ whenever 
$Y(M,\gamma)\leq 0$. Assertion \eqref{consign} now follows from these two observations.

Since we have  now
  observed  that a smooth compact manifold $M$ of dimension $n\geq 3$ has $\mathscr{Y}(M) > 0$ iff $M$ admits some metric $g$ 
of positive scalar curvature, let us  next  recall that 
not every such manifold $M$ has this property. 
Indeed, the first obstruction to the existence of positive-scalar-curvature metrics  on compact manifolds was 
discovered by Lichnerowicz \cite{lic}, who observed  that the Dirac operator $\Dir : \Gamma (\mathbb{S})\to \Gamma (\mathbb{S})$
on  a Riemannian spin manifold satisfies the so-called Weitzenb\"ock formula 
\begin{equation}
\label{lwf} 
{\Dir }^2  = \nabla^* \nabla + \frac{s}{4},
\end{equation}
and therefore has trivial kernel and co-kernel if the scalar curvature $s$ is everywhere positive. However, when 
$n\equiv 0 \bmod 4$, the full spinor bundle decomposes as a Whitney sum $\mathbb{S}= \mathbb{S}_+\oplus \mathbb{S}_-$ of
the so-called chiral  spinor bundles, and the Dirac operator correspondingly decomposes as 
$\Dir= \dir \oplus {\dir}^{\ast}$, where the chiral Dirac operator 
 $$\dir : \Gamma (\mathbb{S}_+)\to \Gamma (\mathbb{S}_-)$$ is an elliptic operator 
whose index $\widehat{A}(M)$ had previously been shown   by Atiyah and Singer \cite{index-bams} 
to be a 
specific   linear combination 
of Pontryagin numbers, and thus   in particular a cobordism invariant.  This allowed Lichnerowicz to 
 prove   that a smooth compact spin-manifold $M^{4k}$ cannot  admit  metrics of positive scalar curvature if
$\widehat{A}(M)\neq 0$. In his thesis, Hitchin \cite{hitharm} then  generalized Lichnerowicz's result 
by noticing that \eqref{lwf} also gives an obstruction to the existence of positive-scalar-curvature metrics on spin 
manifolds of dimension $n\equiv 1$ or $2\bmod 8$. Indeed, in these dimensions 
 there is,  for each spin structure on a smooth compact  manifold, a $\ZZ_2$-valued   invariant  $\alpha$   given by 
$\dim \ker {\not \!\partial} \bmod 2$ when $n\equiv 1 \bmod 8$,   or by $\dim \ker \dir \bmod 2$ when $n\equiv 2 \bmod 8$. 
Since this element of $\ZZ_2$ is independent of the choice of a Riemannian metric $g$ on $M$,
Hitchin  was therefore able to prove  that a necessary condition for the existence of
a positive-scalar-curvature metric is that $\alpha$ must vanish for every spin structure.
When $M$ is simply connected, it can then have at most one spin structure,  so this 
discussion then only involves  an invariant $\alpha (M) \in \ZZ_2$ of any smooth compact simply connected 
manifold of dimension $n\equiv 1$ or $2\bmod 8$ with $w_2=0$. 
To keep the notation  as simple as possible,  one  extends the definition of $\alpha (M)$ 
to  smooth compact spin manifolds $M$ of other dimensions $n$ by setting 
$\alpha (M^n) := \widehat{A}(M)\in \ZZ$ if $n\equiv 0\bmod 4$ and $\alpha (M^n) := 0$ if $n\equiv 3, 5, 6$ or $7\bmod 8$. 
Hitchin's generalization of Lichnerowicz's theorem then tells us  that a simply connected  spin manifold $M$ cannot admit a metric
of positive scalar curvature if $\alpha (M) \neq 0$. Remarkably,  $\alpha(M)$ is actually invariant under spin cobordisms, 
and so only depends on  the spin-cobordism 
class $[M]\in \Omega_n^{\rm Spin}$ of the  $n$-manifold $M$.

The role of spin cobordism in this story means that the obstruction $\alpha (M)$ is invariant under elementary surgeries 
in a suitable   range of dimensions. 
Conversely, 
  Gromov-Lawson \cite{gvln}  and Schoen-Yau \cite{syrger}  independently proved
   that the existence of a positive-scalar-curvature metric  on $M$ is invariant
 under elementary surgeries in codimension $\geq 3$.
Using this, Gromov and Lawson then went on 
  to show that every compact  simply-connected {\em non-spin} manifold of dimension $\geq 5$ admits metrics
of positive scalar curvature, just by showing that every such manifold is obtained by  a sequence of such surgeries
on   products and disjoint unions of specific positive-scalar-curvature generators of the oriented cobordism ring $ \Omega^{\rm SO}$. 
For simply connected spin manifolds, they also conjectured that Hitchin's obstruction $\alpha : \Omega_n^{\rm Spin}\to \mathsf{KO}^{-n}(\mbox{pt})$ was the {\sf only} obstruction to the existence of positive-scalar curvature metrics, and observed that this would
follow from their surgery result  if one could show that $\ker \alpha \subset \Omega^{\rm Spin}$ was generated by spin manifolds of positive scalar curvature.  Stolz \cite{stolz} then proved  this conjecture by  showing that every cobordism class in 
$\ker \alpha$ can actually be represented  by the 
total space of an  $\mathbb{HP}_2$-bundle over spin a manifold. 
Consequently, 
  every  simply connected $n$-manifold $M$, $n \geq 5$, 
   satisfies exactly one of the following:
   \begin{itemize}
   \item either $\mathscr{Y}(M) > 0$; or else 
   \item $M$ is  a spin manifold with 
$\alpha (M) \neq 0$. 
   \end{itemize}

It was Petean \cite{jp3} who  first seriously considered the results of Gromov, Lawson, and Stolz from the perspective  of the Yamabe invariant, and  realized that more could be proved in this context. First, Petean showed that the Gromov-Lawson surgery arguments 
also imply  that, for any $\varepsilon > 0$,  
the condition $\mathscr{Y}(M) > -\varepsilon$ is preserved under elementary surgeries in codimension $\geq 3$. Second, 
Petean discovered that adjoining  a well-chosen collection of Ricci-flat manifolds of special holonomy to 
Stolz's $\mathbb{HP}_2$-bundles  shows that the
spin-cobordism ring $\Omega^{\rm Spin}$ is generated by manifolds with non-negative Yamabe invariant. 
Putting these facts together, he  obtained 

\begin{theorem}[Petean] 
\label{petproject} 
Any  compact simply-connected $n$-manifold $M$, $n\geq 5$, has Yamabe invariant $\mathscr{Y}(M) \geq 0$. 
Moreover, such a manifold has $\mathscr{Y}(M)=0$ iff $M$ is a spin manifold with $\alpha (M) \neq 0$. 
\end{theorem}

On the other hand, since the  $3$-dimensional Poincar\'e conjecture follows \cite{bbbpm,lott} from 
Perelman's proof \cite{perelman1} of  Thurston's geometrization conjecture,
 any  simply-connected compact $3$-manifold $M^3$ is necessarily 
 diffeomorphic to $S^3$, and
therefore has  $\mathscr{Y}(M)>0$.  In conjunction with  Theorem \ref{petproject} and assertion \eqref{consign}, this  then immediately 
implies Theorem \ref{tour}. 

For simply connected manifolds of dimension $n\neq 4$, Theorem \ref{petproject} provides a 
complete understanding of the  {\em sign} of the Yamabe invariant, but usually says  nothing at all about its 
precise value.  On the other hand, \eqref{ceiling} gives us a universal  upper bound,
while  Obata's theorem provides a non-trivial lower bound for $\mathscr{Y}(M)$ whenever $M$ admits an 
Einstein metric of positive scalar curvature. In conjunction with Kobayashi's inequality \cite{okob} 
\begin{equation}
\label{kobble}
\mathscr{Y}(M^n)\geq  0, ~\mathscr{Y}(N^n) \geq  0\quad \Longrightarrow \quad \mathscr{Y}(M\# N) \geq   \min [\mathscr{Y}(M),\mathscr{Y}(N)]
\end{equation}
this  confines  the Yamabe invariants of   many
 manifolds to   specific (and often reasonably   narrow) ranges. Of course,  \eqref{kobble} is a statement
 about surgery in codimension $n$, but generalizing it to surgeries in other codimension involves extra complications. 
 In this direction, the best available analogue of the  Gromov-Lawson-Petean  surgery result
 is a theorem of 
  Ammann, Dahl, and Humbert \cite{ADH}, which 
says that, for every $n$,  there is a constant $\Lambda_n > 0$ such  that, whenever  $\varepsilon \leq  \Lambda_n$, the condition 
$\mathscr{Y}(M^n) > \varepsilon$ is invariant under elementary surgeries in codimension $\geq 3$.  One consequence is  
the following 
{\sf gap theorem}: for each    $n \geq 5$,  there is a $\delta_n > 0$ such that  every compact simply-connected $n$-manifold $M$ with 
$\mathscr{Y}(M) > 0$ actually satisfies $\mathscr{Y}(M^n) > \delta_n$.

It must however  be emphasized that   simple-connectivity plays a key  role  in the proofs of 
Theorems \ref{tour} and \ref{petproject}, 
and that the story is   known to   become far   more complicated when the fundamental group is non-trivial. 
For example, Perelman's geometrization of $3$-manifolds
involves  monotonicity results for the 
Ricci-flow \cite{perelman1} that allow one to compute the Yamabe invariants of many $3$-manifolds, owing to   the intimate 
relationship between Perelman's  $\bar{\lambda}$ invariant 
and the Yamabe invariant \cite{AIL}. Indeed, 
 Anderson \cite{ander34} showed that  Perelman's results imply 
  that  the hyperbolic metric $h_0$ on any compact
quotient  $M^3= \mathcal{H}^3/\Gamma$ of hyperbolic $3$-space 
  realizes the Yamabe invariant, in the sense 
that $\mathscr{Y}(M) = \mathscr{E}(M, h_0)$;   thus, in particular, 
 $\mathscr{Y}(M) < 0$ and $\mathcal{I}_s (M) > 0$ for any hyperbolic $3$-manifold $M$. On the other hand,  geometrization also tells us that 
 that any compact $3$-manifold $M^3$ with $|\pi_1(M^3) |<\infty$ is a 
  spherical space-form $S^3/\Gamma$, $\Gamma \subset \mathbf{SO}(4)$, 
and therefore  has $\mathscr{Y}(M) > 0$.  In this context,   an  argument due to Bray and Neves \cite{brayneves} 
 shows that the Yamabe invariant of  $\mathbb{RP}^3 = S^3/\ZZ_2$ is actually   realized by the constant-curvature metric, 
so that $\mathscr{Y}(\mathbb{RP}^3 ) = \mathscr{Y}(S^3)/2^{2/3} = 6\pi^{4/3}$. While this method unfortunately does not allow one to similarly 
compute the Yamabe invariant for   other 
spherical space-forms, 
it does  at least  yield   \cite{akuneves}
  a beautiful  {\sf gap theorem}: if    $\mathscr{Y}(M^3) < \mathscr{Y}(S^3)$,  then  $\mathscr{Y}(M^3) \leq \mathscr{Y}(\mathbb{RP}^3 )$.

Related  issues persist in higher dimensions. While
 Gromov, Lawson, and Stolz  showed, for  $n \geq 5$,   that the Dirac operator  is the source of  the only 
obstruction to a  simply-connected
$n$-manifold  $M$ having $\mathscr{Y}(M^n)>0$,   this precept  no longer  holds true  when $\pi_1 (M) \neq 0$. 
Perhaps the first  indication that something else might be at play was provided by the discovery that the $n$-torus 
$\mathbb{T}^n$ does not admit metrics of positive scalar curvature, even though its Dirac operator 
has index zero. Indeed, notice that, by mutiplying by $S^1$ if necessary,   it suffices to prove it when $n=2m$ is even.
Now, while $\alpha (\mathbb{T}^{2m}) = 0$,  there are nonetheless spin$^c$ Dirac operators of non-zero index on any covering $\mathbb{T}^{2m}\to \mathbb{T}^{2m}$,
and that one can arrange for the curvature of the line bundle to which   the spin$^c$ Dirac operator is coupled to 
be the pull-back of an arbitrarily small multiple of the standard symplectic form simply by 
 choosing a suitable covering of 
 sufficiently high degree;  thus, a variation on the Lichnerowicz argument therefore shows that $\mathbb{T}^{2m}$
 cannot admit a metric of scalar curvature $s\geq \varepsilon$ for any $\varepsilon > 0$. However, versions of 
 this argument works equally well if we instead couple the Dirac operator to vector bundles of higher rank, just
 as long as we can ensure that  of the coupled operator has non-zero index, and that the bundle curvature can be chosen to 
 converge point-wise to  zero as we pass to covers of higher and higher degree. 
 By systematizing this idea,  Gromov and Lawson \cite{gvln2} were able to more generally prove the non-existence of 
positive-scalar-curvature metrics  on {\sf enlargeable}  manifolds. 
Here an $n$-manifold $M$ is said to be enlargeable if for some  Riemannian metric $g$
and every  $\rad > 0$, 
 there is 
 a finite spin covering 
$\widetilde{M}\to M$ that, when equipped with pull-back $\tilde{g}$ of $g$,  
 admits a distance-decreasing map of positive  degree to the standard $n$-sphere of radius $\rad$. 
 This condition turns out  to be  to be  metric-independent, and moreover   depends only  on  the  homotopy-type of $M$. 
 While we will find these results on 
 enlargeable manifolds very useful  in \S \ref{yzero}, an inherent limitation of the method is  that it  only applies to 
  manifolds whose universal covers are spin.

However, around the same time, Schoen and Yau \cite{syrger} discovered an entirely different method  
 that, for example, also    proves the non-existence
of positive-scalar-curvature metrics on manifolds like $\mathbb{T}^5\# [\mathbf{SU}(3)/\mathbf{SO}(3)]$, where the universal cover 
has $W_3\neq 0$, and so is not even spin$^c$.  Since will also  make essential use of this technique in \S \ref{yzero},   we will now carefully review 
the Schoen-Yau method in the context of  manifolds of dimension $n\leq 7$. For  recent progress on  extending these arguments
    to higher
dimensions, see Schoen and Yau's  recent preprint\cite{schoenyau17}.

The Schoen-Yau method depends in part on Jim Simons' second-variation formula  for a minimal hypersurface. 
Let $(\Sigma^m , h) \subset (M^{m+1}, g)$, $m \geq 2$,  be  a compact  oriented 
 minimal hypersurface in an oriented Riemannian 
 manifold, and let  $\Sigma_t \subset  M$  be any smooth  $1$-parameter variation of $\Sigma=\Sigma_0$ with 
with normal variation vector vector field $\mathsf{v} = u \mathsf{n}$, where $\mathsf{n}$ is the unit normal vector of $\Sigma$.  
The second-variation formula \cite[Theorem 3.2.2]{simons} then asserts that 
the $m$-dimensional volume $\mathscr{A}(t)$ of $\Sigma_t$ satisfies 
\begin{equation}
\label{jim} 
\mathscr{A}^{\prime \prime} (0) = \int_\Sigma \left[|\nabla u|^2 -r_g(\mathsf{n}, \mathsf{n}) u^2 - |\gemini|^2u^2\right]\, d\mu_h 
\end{equation}
where $r_g$ is the  Ricci tensor of the ambient metric and $\gemini$ is the 
 second fundamental form of  $\Sigma \subset M$. However,
the Gau{ss}-Codazzi equations
imply  that  the scalar curvatures of $h$ and $g$ are related  along $\Sigma$ by 
 $$s_h= s_g - 2r_g (\mathsf{n} , \mathsf{n} )+H^2 - |\gemini |^2,$$
 where  the mean curvature  $H = h^{ij} \gemini_{ij}$ of $\Sigma$ 
vanishes in our case because   $\Sigma$ is  assumed to be minimal. Consequently, 
  \eqref{jim} can  be rewritten as 
\begin{equation}
\label{jimmy}
 \int_\Sigma \left[2|\nabla u|^2 + 
s_h u^2\right]\, d\mu_h= 
2\mathscr{A}^{\prime \prime} (0)  + \int_\Sigma (s_g + |\gemini|^2)u^2\, d\mu_h
\end{equation}
for every $1$-parameter variation of $\Sigma$.

If we now   assume   that  $(M,g)$ has positive scalar curvature $s_g > 0$ and that 
$\Sigma\subset M$ is  {\sf volume-minimizing} in its homology class, it then follows that 
$\Sigma$ carries a positive-scalar-curvature metric $\widehat{h}$ conformal to $h$. 
Indeed, since our  volume-minimizing hypothesis on $\Sigma$ 
  forces $\mathscr{A}^{\prime \prime} (0) \geq 0$
for any $1$-parameter variation,  plugging 
the positivity of $s_g$  into   \eqref{jimmy}  now forces  
\begin{equation}
\label{james}
 \int_\Sigma \left[2|\nabla u|^2 + 
s_h u^2\right]\, d\mu_h  > 0
\end{equation}
for any smooth function $u\not \equiv 0$. 
If $m\geq 3$,  now let $\widehat{h}= u^{p-2}h$  be a Yamabe metric, where 
 $p = \frac{2m}{m-2}>2$, and   notice that \eqref{james} then  tells us that 
$$  Y(\Sigma , [h]) = \frac{\int_\Sigma [(p+2) |\nabla u |^2 + s_hu^2 ]d\mu_h}{\| u\|_{L^p}^2} > 0, $$
which shows  that  $(\Sigma, \widehat{h})$     has positive scalar curvature. 
On the other hand, if  $m=2$, setting  $u\equiv 1$  in  \eqref{james} yields  $\chi (\Sigma ) > 0$ by  Gauss-Bonnet, so one therefore has  
 $(\Sigma , [h])\cong (S^2, [g_0])$ by classical uniformization.

The Schoen-Yau strategy now proceeds by downward induction on the dimension of the manifold. Suppose a smooth compact
oriented $n$-manifold, $n\leq 7$, admits a metric of positive scalar curvature, and  let $\mathsf{a} \in H^1 (M, \ZZ)$ be a non-trivial 
cohomology class. Compactness results in geometric measure theory   \cite[\S 5.1.6]{federer} guarantee that there is a mass-minimizing rectifiable current that
represents the Poincar\'e dual homology class $\mathsf{a}\in H_{n-1}(M, \ZZ )$, and our assumption that $n\leq 7$ then guarantees,
by a regularity  result that  Federer \cite[Theorem 5.4.15]{federer}  deduced from  a  lemma of  Simons \cite[Lemma 6.1.7]{simons}, 
 that this current is moreover   a sum of disjoint smooth compact oriented hypersurfaces, with 
positive integer multiplicities. Any one of these hypersurfaces $\Sigma^{n-1} \subset M^n$ then admits 
a metric of positive scalar curvature by the above argument, and one may then try to repeat the same argument to 
further reduce the dimension. 

For example, if there is a map $M^n \to \mathbb{T}^n$ of non-zero degree, pulling back the generators
of $H^1(\mathbb{T}^n, \ZZ)$ yields a collection of classes
$\mathsf{a}_1, \ldots , \mathsf{a}_n \in H^1 (M, \ZZ)$ with $\mathsf{a}_1 \cup \cdots \cup  \mathsf{a}_n \neq 0$.
We now  apply  the above argument to 
 $\mathsf{a}=\mathsf{a}_n$, and  then  choose the hypersurface $\Sigma^{n-1}$
so that $\langle \mathsf{a}_1 \cup \cdots \cup  \mathsf{a}_{n-1} , [\Sigma^{n-1} ]\rangle \neq 0$.
Iterating the same argument with respect to conformal rescalings of the induced metrics  then  produces a compact surface $\Sigma^2\subset \Sigma^3 \subset \cdots \subset \Sigma^{n-1}$ with 
positive scalar curvature and $\langle \mathsf{a}_{1} \cup   \mathsf{a}_2 , [\Sigma^{2} ]\rangle \neq 0$,
contradicting the fact that   $\chi (\Sigma^2) > 0$ implies $b_1(\Sigma^2)=0$. Thus, whenever a smooth compact
oriented $n$-manifold $M$, $n\leq 7$, admits a map $M\to \mathbb{T}^n$ of non-zero degree, it necessarily satisfies $\mathscr{Y}(M)\leq 0$. As a corollary, 
we thus deduce  that examples  like the previously-mentioned 
 connected sum $M^5= \mathbb{T}^5 \# [\mathbf{SU}(3)/\mathbf{SO}(3)]$ 
cannot admit metrics of positive scalar curvature. 

While the above discussion emphasizes the use of homologically volume-minimizing hypersurfaces, 
 the Schoen-Yau method is also applicable  in contexts where one can prove the existence of  stable minimal hypersurfaces
 for other reasons. 
 As we will also see in \S \ref{yzero},   the inductive step in the Schoen-Yau approach
 can  also be applied in substantially different ways. Variations on this approach thus  allow one 
  to prove the non-positivity of the Yamabe invariant for many   interesting manifolds with large fundamental group
 that  do not neatly conform to  the  simple paradigm we've just described.

\section{Dimension Four:  Yamabe Negative Case}
\label{yneg}

Having  carefully discussed the behavior of the  Yamabe invariant in other dimensions, we now shift our focus to 
dimension four, where we will see that 
the relationship between 
scalar curvature and differential topology is strangely   
  different.
  
 The peculiar character of $4$-dimensional geometry is 
 largely ascribable to a single Lie-theoretic fluke: while the rotation group  $\mathbf{SO}(n)$ is  a {simple} Lie group
 for every other $n\geq 3$, this fails in dimension four. Indeed, since 
 $$\mathbf{Spin}(4)  =  \mathbf{Sp}(1) \times  \mathbf{Sp}(1) \cong \mathbf{Spin}(3) \times  \mathbf{Spin}(3), $$
  the adjoint action of $\mathbf{SO}(4)$ on $\mathfrak{so}(4)$ is consequently   reducible:
 \begin{equation}
\label{al-deco}
\mathfrak{so}(4)\cong  \mathfrak{so}(3)\oplus \mathfrak{so}(3).
\end{equation}
 This has an immediate and  powerful impact on the geometry of $2$-forms, 
because, for any $n$,    ${\mathfrak s \mathfrak o}(n)$ and  $\Lambda^2({\mathbb R}^n)$ 
are   isomorphic as $\mathbf{SO}(n)$-modules. Thus, the decomposition \eqref{al-deco}  implies  that 
 the rank-6 bundle of 2-forms on an oriented Riemannian 4-manifold $(M,g)$ 
    invariantly  decomposes as  the Whitney sum of 
 two rank-3 bundles
\begin{equation} 
\Lambda^2 = \Lambda^+ \oplus \Lambda^- ,
\label{f-deco} 
\end{equation}
a phenomenon without analogue in  other   dimensions. 
Here the summands $\Lambda^\pm$  just  turn out to be  the the $(\pm 1)$-eigenspaces of the Hodge star
operator 
\begin{equation}
\label{starring}
\star: \Lambda^2 \to \Lambda^2,
\end{equation}
and, for this reason,  $\Lambda^+$ is  called  the bundle of {\sf self-dual} $2$-forms, 
while   $\Lambda^-$ is called the bundle of {\sf anti-self-dual} $2$-forms.
Of course, the distinction between the two depends on a choice
of orientation; reversing the orientation of $M$ simply interchanges $\Lambda^+$ and $\Lambda^-$.

On any oriented Riemannian $4$-manifold $(M,g)$,
the bundle $\Lambda^+\to M$
 carries a natural  inner product and orientation, so 
 every fiber of its unit sphere bundle $Z = S(\Lambda^+)$ carries both a metric and 
orientation. This allows us to consider the so-called {\em twistor space}  $Z$ as a bundle of complex projective lines  $\CP_1$. 
Remarkably, this $\CP_1$-bundle  can  always be realized 
\begin{equation}
\label{presigma}
\mathbb{P}(\mathbb{V}_+)  = S(\Lambda^+)
\end{equation}
as the projectivization of 
 a rank-$2$ complex vector bundle $\mathbb{V}_+\to M$. What's more,  the choice of such a  $\mathbb{V}_+$ 
 is  equivalent  to choosing a {\sf spin$^c$ structure} on $M$.
 This  stems from  the fact that $Z$ can naturally  be expressed as
 \begin{equation}
\label{ahas}
S(\Lambda^+) = \mathfrak{F}/\mathbf{U}(2),
\end{equation}
 where $\mathfrak{F}$ is the principal $\mathbf{SO}(4)$-bundle of oriented orthonormal frames.

 Indeed, according to the usual definition, a spin$^c$ structure on $(M, g)$ is a choice of  principal 
 $\mathbf{Spin}^c (4)$-bundle
$\widehat{\mathfrak{F}}\to M$, 
where 
$$\mathbf{Spin}^c (4):= [\mathbf{Spin} (4)\times \mathbf{U}(1)]/\ZZ_2 = [\mathbf{Sp} (1)\times \mathbf{Sp} (1)\times \mathbf{U}(1)]/\langle (-1,-1,-1)\rangle ,$$
together with a fixed  isomorphism $\mathfrak{F} =\widehat{\mathfrak{F}}/\mathbf{U}(1)$. 
Up to isomorphism,  such a structures is  determined by the Chern class $\widehat{\mathfrak{c}}\in H^2 (\mathfrak{F},\ZZ)$
of the circle bundle $\widehat{\mathfrak{F}}\to \mathfrak{F}$, and  $\widehat{\mathfrak{c}}$   can in principle   be  any element of $H^2 (\mathfrak{F},\ZZ)$ 
whose restriction to a fiber yields  the non-trivial element of $H^2 (\mathbf{SO}(4),\ZZ) \cong \ZZ_2$. 
On the other hand, expressing $Z$ as $\mathbb{P}(\mathbb{V}_+)$ gives rise to  an $\mathcal{O}(1)$ line-bundle  $\mathscr{L}\to Z$,
and so gives us a cohomology class $\mathfrak{c}=c_1(\mathscr{L})\in H^2 (Z, \ZZ)$ that satisfies $\langle \mathfrak{c} , [S^2] \rangle = 1$, 
 where $[S^2]\in H_2(Z,\ZZ)$ is the
    the fiber homology class. This allows us to associate a unique spin$^c$ structure to any choice of $\mathbb{V}_+$
    by setting $\widehat{\mathfrak{c}} = {\zap q}^*\mathfrak{c}$, where ${\zap q}: \mathfrak{F} \to \mathfrak{F}/\mathbf{U}(2)$
    is the quotient map. Conversely, one can construct $\mathbb{V}_+$ from a principal 
 $\mathbf{Spin}^c (4)$-bundle
$\widehat{\mathfrak{F}}\to M$
    by applying the associated bundle construction to the  representation of $\mathbf{Spin}^c (4)$ on $\CC^2$ 
   $$ [\mathbf{Sp} (1)\times \mathbf{Sp} (1)\times \mathbf{U}(1)]/\langle (-1,-1,-1)\rangle
   \longrightarrow [\mathbf{Sp} (1)\times \mathbf{U}(1)]/\langle (-1,-1)\rangle = \mathbf{U}(2)$$
  gotten by dropping the second $\mathbf{Sp} (1)$ factor. Of course, dropping the first $\mathbf{Sp} (1)$ factor
 instead  gives us a second rank-$2$ complex vector bundle $\mathbb{V}_-$ with 
  $$\mathbb{P}(\mathbb{V}_-)  = S(\Lambda^-).$$
  The relationship between these two representations  then guarantees that 
  \begin{equation}
\label{bigdog}
\Hom (\mathbb{V}_+ , \mathbb{V}_-) =  \CC \otimes T^*M , 
\end{equation}
while  the Hermitian line-bundle $L$ associated to the representation 
 $$ [\mathbf{Sp} (1)\times \mathbf{Sp} (1)\times \mathbf{U}(1)]/\langle (-1,-1,-1)\rangle
   \longrightarrow  \mathbf{U}(1)/\langle -1\rangle = \mathbf{U}(1)
 $$
 automatically satisfies 
\begin{equation}
\label{lineup}
L:= \wedge^2 \mathbb{V}_+= \wedge^2 \mathbb{V}_- .
\end{equation}
This in particular makes $\widehat{\mathfrak{F}}$ into a double cover of the fiber-wise product  $\mathfrak{F}\oplus S(L)$
of the oriented  Riemannian frame-bundle and the unitary frames for $L$, so any  $\mathbf{U}(1)$ connection $\theta$
on $L$ induces a uniquely-defined principal $\mathbf{Spin}^c(4)$-connection on $\widehat{\mathfrak{F}}$,
and this in turn  induces Hermitian connections $\nabla_\theta$ on the bundles $\mathbb{V}_\pm$. From these, 
one can of course recover   the Riemannian  connection on $TM$  via the isomorphism  \eqref{bigdog}.

Now, using the Gysin sequence and Poincar\'e duality,  the Euler class $\mathsf{e} (\Lambda^+)\in H^3(M,  \ZZ)$ can be
shown to vanish  for any compact oriented Riemannian $4$-manifold $(M,  g)$, because a  twistor-lift construction 
 allows one to show  
that the  torsion subgroup $\mathfrak{T}_2(Z)\subset H_2(Z, \ZZ)$ surjects onto  
the torsion subgroup  $\mathfrak{T}_2(M)\subset H_2(M, \ZZ)$ under  the twistor  projection 
$Z \to M$, and that  the torsion subgroup $\mathfrak{T}^3(M)\subset H^3(M, \ZZ)$ therefore injects into $\mathfrak{T}^3(Z)\subset H^3(Z, \ZZ)$.
 This is of course in perfect agreement with  the conventional approach \cite{hitharm,lawmic} 
 to the existence of spin$^c$ structures, because $\mathsf{e}(\Lambda^+)$ 
can also be shown  {\em a priori} to coincide with the   third {\em  integral} Stiefel-Whitney class 
$W_3(TM)\in H^3 (M, \ZZ)$,
and the latter can in turn be shown to vanish for any oriented $4$-manifold 
by an argument due to Hirzebruch and Hopf \cite{hiho}. 
Thinking of a spin$^c$ structure as an element  $\mathfrak{c}\in H^2(Z, \ZZ)$ with fiber integral $+1$  also
gives such structures  a metric-independent meaning, because the $2$-sphere bundles associated with any two 
metrics are naturally bundle-equivalent. Moreover,  the Gysin sequence of $\Lambda^+\to M$ 
implies that $H^2(M, \ZZ)$ acts freely and transitively on the set of spin$^c$ structures by pull-back;
and, in terms of the vector bundles  $\mathbb{V}_\pm\to M$,  the effect  of this is that 
$$
\mathbb{V}_+ \rightsquigarrow  E\otimes \mathbb{V}_+ , \quad \mathbb{V}_-\rightsquigarrow  E\otimes \mathbb{V}_- , 
$$
as $E$ ranges over all  complex line-bundles $E\to M$.

One consequence of \eqref{ahas} is that $Z$ can be  canonically identified \cite{AHS} with the set of all 
metric-compatible point-wise almost-complex structures that also determine the given orientation of
$M$. Concretely, this amounts to the observation  that if $J\in \End (T_pM)$, $J^2=-I$,   is  an almost-complex structure  
at $p$ that preserves  $g$ and determines  the fixed orientation, and if we then associate a tensor  $\omega$ with $J$
via  the prescription 
$$\omega (\cdot , \cdot ) = g(J\cdot , \cdot ),$$
then $\omega /\sqrt{2}$ is a unit-norm self-dual $2$-form, and every unit-norm self-dual form at $p$ conversely 
arises  this way from a unique almost-complex structure $J$ at $p$.  
This gives rise to yet another way to understand spin$^c$ structures. Indeed, for any given spin$^c$ structure, the vector bundle
$\mathbb{V}_+$ has real rank equal to the dimension of $M$, so  one can  find smooth sections of $\mathbb{V}_+$ that 
are non-zero everywhere except at a chosen base-point $q\in M$ for our connected oriented compact $4$-manifold.
Since this gives us a section of $\mathbb{P} (\mathbb{V}_+)$ on the complement of the base-point, 
 equations   \eqref{presigma} and \eqref{ahas} together 
give the punctured manifold $M-\{ q\}$  an associated almost-complex structure. On the other hand, 
it is a standard fact \cite{lawmic} that  an  almost-complex structure on 
$M-\{ q\}$  determines a spin$^c$ structure on $M-\{ q\}$, and  it is also easy to see that any such spin$^c$ structure on $M-\{ q\}$
 uniquely 
extends to a spin$^c$ structure on all of $M$. These two constructions are actually 
inverses of each other, so the  spin$^c$ structure determined by the constructed almost-complex structure on the punctured manifold
is  exactly the one we started with.  Thus,  a spin$^c$ structure on $M$ may also be thought of
as an equivalence class of  almost-complex structures on the punctured manifold $M-\{ q\}$, where $q\in M$ is an arbitrary  
base-point.  Since \eqref{bigdog} moreover allows us to identify 
the restriction of $\mathbb{V}_-$ to $M-\{ q\}$ 
with the $(1,0)$-tangent bundle of $J$,  the line-bundle $L=\det \mathbb{V}_-$ restricted to $M-\{ q\}$
is actually the anti-canonical line bundle of the almost-complex structure $J$. Since $H^2(M, \ZZ ) = H^2(M-\{ q\}, \ZZ )$,
this in particular shows that 
\begin{equation}
\label{lift}
c_1(L) \equiv w_2 (M) \bmod 2 . 
\end{equation}
Also notice that the action of $H^2 (M, \ZZ)$ on spin$^c$ structures is manifested in this context by
$$L\rightsquigarrow L  \otimes E^2 , $$
so that  \eqref{lift} is actually the only  constraint on $c_1(L) = c_1 (\mathbb{V}_+)$. Indeed, if 
(but only if) 
$H_1(M, \ZZ)$ has trivial $2$-torsion,  spin$^c$ structures are then  in one-to-one correspondences with  integer
cohomology classes  $c_1(L)$ satisfying \eqref{lift}.

While the  above discussion certainly shows that \eqref{presigma}    has significant  geometric consequences, the 
 intimate relationship  between the bundles $\mathbb{V}_+$ and 
$\Lambda^+$  has a more  vivid  realization that is ultimately  far more   consequential. Indeed, \eqref{presigma}
can be enriched into a natural real-quadratic map 
$$\sigma : \mathbb{V}_+ \to \Lambda^+$$
that plays a key role in the theory of spin$^c$ Dirac operators. 
Indeed, since  $\mathbb{V}_+$ and $\Lambda^+$ are,  respectively,  the bundles associated with  the 
defining and 
adjoint representations of $\mathbf{U}(2)= \mathbf{Spin}^c(4)/(\mathbf{1} \times \mathbf{Sp}(1)\times \mathbf{1})$, 
the relationship between these representations gives us a (Clifford multiplication)  isomorphism 
\begin{equation}
\label{canine}
\End_0 (\mathbb{V}_+) = \CC\otimes \Lambda^+,
\end{equation}
 where $\End_0$ indicates  trace-free endomorphisms, and where 
complex conjugation in $\CC\otimes \Lambda^+$ corresponds to 
 $A\mapsto -A^*$ in $\End_0 (\mathbb{V}_+)$. Thus,  $\Lambda^+$ is identified with 
the  skew-adjoint trace-free endomorphisms of  $\mathbb{V}_+$, while $i\Lambda^+$ 
is  identified with 
the  self-adjoint trace-free endomorphisms. For any $\Phi\in \mathbb{V}_+$, we   may thus uniquely define $\sigma (\Phi )\in \Lambda^+$
by declaring  that the corresponding trace-free  skew-adjoint  map $\mathbb{V}_+\to \mathbb{V}_+$ is to  be given by 
$$\sigma (\Phi)\cdot  \Psi =  i \left[ \langle \Psi , \Phi \rangle \Phi - \frac{1}{2} |\Phi|^2 \Psi \right].$$
With this convention, one then has 
\begin{equation}
\label{radial}
|\sigma (\Phi )| = \frac{|\Phi|^2}{2\sqrt{2}}
\end{equation} 
for any $\Phi\in \mathbb{V}_+$, and 
\begin{equation}
\label{clef} 
\langle  \omega  \cdot \Phi , \Phi \rangle = 4i \langle \omega , \sigma (\Phi ) \rangle 
\end{equation}
for any $\omega \in \Lambda^+$.

Now, given a spin$^c$ structure, we have already observed that every 
unitary 
connection $\theta$ on $L$, in conjunction with the Riemannian  connection on $TM$,  induces a unitary connection 
$$\nabla_{\theta} : \Gamma ({\mathbb V}_{+})\to \Gamma (\Lambda^1\otimes {\mathbb V}_{+})$$
on $\mathbb{V}_+$. 
Composing this with the (Clifford multiplication)  map
$$\Lambda^1\otimes {\mathbb V}_{+}\to {\mathbb V}_{-}$$
induced by \eqref{bigdog} then 
gives us  \cite{hitharm,lawmic} a 
spin$^c$ analogue  
$$\dir_\theta \! : \Gamma ({\mathbb V}_{+})\to \Gamma ({\mathbb V}_{-})$$
of the chiral Dirac operator. 
In this setting, the Lichnerowicz Weitzenb\"ock formula  \eqref{lwf} then generalizes \cite{hitharm,lawmic} 
as 
$$\dir_{\theta}^{\ast}\! \dir_{\theta} = \nabla_\theta^*\nabla_\theta + \frac{s}{4} - \frac{1}{2} F_\theta^+\cdot $$ 
where  the self-dual part  $F_{\theta}^{+}\in i  \Lambda^+$ of the curvature of $L$ 
  acts on $\mathbb{V}_+$ via \eqref{canine}.  In light of  \eqref{clef}, we therefore have 
\begin{equation}
\label{wtw} 
\langle \Phi , \dir_{\theta}^{\ast}\! \dir_{\theta} \Phi  \rangle = \frac{1}{2}\Delta |\Phi |^2 + |\nabla_{\theta} \Phi |^2 + 
\frac{s}{4} |\Phi |^2 + 2 \langle -iF_{\theta}^{+} , \sigma (\Phi ) \rangle 
\end{equation}
for any  $\Phi\in \Gamma ({\mathbb V}_+)$. It  is this last   form of the spin$^c$ Weitzenb\"ock formula that we 
will generally   use in what follows. 

On  a spin manifold, the 
standard Dirac operator is entirely determined by the relevant   Riemannian metric, and we previously saw in 
\S \ref{yammer} that \eqref{lwf} then implies a  non-existence result for 
positive-scalar-curvature metrics on spin manifolds $M$ with $\alpha (M)\neq 0$.   
However, this approach does not obviously  generalize to the spin$^c$ setting, because the curvature
of the connection $\theta$ on the  line bundle $L$ appears  in \eqref{wtw},
and is in principle entirely independent of the Riemannian geometry of $(M,g)$. 
To overcome this, is therefore necessary to impose  conditions on $\theta$ that somehow  tie it 
more closely to the   Riemannian geometry of $(M,g)$. One interesting  choice, which we will discuss in 
detail in \S \ref{ypos} below,   is to simply demand that the curvature 
of $\theta$ be a  harmonic $2$-form. However, a more subtle condition, originally  introduced  by 
Witten \cite{witten}, instead requires that $\Phi$ and $\theta$ together  solve  a coupled 
system of PDE. 

Given a spin$^c$ structure on a compact oriented Riemannian $4$-manifold $(M,g)$, the {\sf Seiberg-Witten equations} thus ask for 
  the Hermitian connection $\theta$ on $L$ and the generalized spinor $\Phi \in \Gamma (\mathbb{V}_+)$ to satisfy 
 the coupled equations 
 \begin{eqnarray} \dir_\theta\Phi &=&0\label{drc}\\
 F_\theta^+&=&i \sigma(\Phi) \label{sd}\end{eqnarray}
 and a solution $(\Phi, \theta )$ of this system is  then called {\sf irreducible} if  $\Phi \not\equiv 0$. 
 However,  \eqref{radial} and \eqref{wtw} together imply that  $\Phi$  satisfies  
  \begin{equation}
 0=	2\Delta |\Phi|^2 + 4|\nabla_{\theta}\Phi|^2 +s|\Phi|^2 + |\Phi|^4 ,	
 	\label{wnbk}
 \end{equation}
so one immediately sees that the Seiberg-Witten equations  (\ref{drc}--\ref{sd}) cannot admit an irreducible   solution 
relative to a metric $g$ with  $s >0$. But why would one ever expect for such solutions to exist? 
The Seiberg-Witten  system is non-linear, so one certainly cannot merely rely on an index calculation
to predict the existence of  solutions. Instead,  Witten had the  remarkable insight that one can instead define a 
new invariant of a smooth compact oriented $4$-manifold with  fixed  spin$^c$  structure by  ``counting'' solutions
of these equations,   in a  manner that can then be shown to be metric-independent. 

However,  whenever the solution space of the Seiberg-Witten equations  is non-empty, it is automatically infinite-dimensional, 
because 
 the {\em gauge group}  
$$\mathscr{G} = \{ \mbox{smooth maps } {\zap f} : M\to S^1\subset \CC \}$$
of circle-valued functions 
 acts on solutions of   (\ref{drc}--\ref{sd}) by 
 $$(\Phi , \theta ) \longmapsto ({\zap f}\Phi , \theta + 2d\log {\zap f}),$$
 thereby carrying solutions of (\ref{drc}--\ref{sd}) into new solutions that 
 essentially just differ by automorphisms of $L$, and so are  
  geometrically really
 the same. Given a spin$^c$ structure $\mathfrak{c}$ on $M$, we are  thus led to consider, for each  Riemannian metric $g$,
 the Seiberg-Witten {\em moduli space} 
 \begin{equation}
\label{moduli}
 \mathfrak{M}_\mathfrak{c} (g) = \{ \mbox{solutions of  (\ref{drc}--\ref{sd})}\}/ \mathscr{G}, 
\end{equation}
and this moduli space can always  be shown to be compact. 
However, $ \mathfrak{M}_\mathfrak{c} (g)$ 
is not  necessarily a manifold, for reasons we will need to overcome in order 
 to define Witten's invariant. 

Fortunately, however, 
there is a simple remedy for this difficulty whenever 
 $b_+(M)\neq 0$. Here $b_+(M)$ 
is the oriented homotopy invariant of $M$ that may be 
calculated by diagonalizing  the {\em intersection form} 
\begin{eqnarray*}
\bullet  :
H^{2}(M, {\mathbb R})\times H^{2}(M, {\mathbb R})	
 & \longrightarrow & ~~~~ {\mathbb R}  \\
	( ~ [\varphi ] ~ , ~ [\psi ] ~) ~~~~~ & 
	\mapsto  & \int_{M}\varphi \wedge \psi ~.
\end{eqnarray*}
of $M$ with real coefficients, and then counting 
$$\left[ 
	  \begin{array}{rl}  \underbrace{
	 \begin{array}{ccc}
	   		 1 &  \qquad &	 \quad  \\
	   		  &	\ddots &   \\
	   		  &	 & 1
	   	  \end{array}}_{b_{+}(M)}
	   	   & 
	   		   \\ 
	   {\scriptstyle b_{-}(M)}  \!
	    \left\{\begin{array}{r}
	    \\
	    \\
	    \\
	    \\
	    \end{array}
	    \right. \! \! \! \! \! \! 
	    & \begin{array}{ccc}
	   		 -1	&  \qquad &  \quad \\
	   		  &	\ddots &   \\
	   		  &	 & -1
	   	  \end{array}
	    \end{array} 
	    \right] ~.
$$
the number of positive directions of $\bullet$ in
the second cohomology;  similarly,   the number of negative directions
is called $b_-(M)$, and   we automatically have $b_+(M) + b_-(M)=b_2(M)$ because 
Poincar\'e duality  guarantees that the intersection pairing $\bullet$
is non-degenerate. This may be made more concrete by identifying $H^2(M, \RR)$ with 
the space of harmonic $2$-forms 
$$\mathcal{H}^2_g =\{ \varphi \in \Gamma (\Lambda^2) ~|~
d\varphi = 0, ~ d\star \varphi =0 \} $$
via the Hodge Theorem; since the Hodge star operator $\star$ defines an involution 
of the right-hand side,  decomposing $\mathcal{H}^2_g$ into the $(\pm1)$-eigenspaces 
of $\star$  then yields 
\begin{equation}
\label{deco-harm} 
\mathcal{H}^2_g = {\mathcal H}^+_{g}\oplus {\mathcal H}^-_{g},
\end{equation}
where
$${\mathcal H}^\pm_{g}= \{ \varphi \in \Gamma (\Lambda^\pm) ~|~
d\varphi = 0\} $$
is  the space of closed (and hence harmonic) self-dual  (respectively, anti-self-dual)  $2$-forms.
We then have 
$$
b_\pm (M) = \dim {\mathcal H}^\pm_{g}, 
$$
because $\bullet$ may concretely be diagonalized, in the  above manner, 
by  choosing an  $L^2$-orthonormal basis for ${\mathcal H}^2_{g}$
consisting of  an  $L^2$-orthonormal basis for ${\mathcal H}^+_{g}$, followed by 
an $L^2$-orthonormal basis for ${\mathcal H}^-_{g}$. 

With this in mind, we now 
 generalize the Seiberg-Witten equations by replacing  \eqref{sd} 
 with the ``perturbed'' equation 
\begin{equation}
\label{ptsd}
iF_\theta^++\sigma(\Phi) =\eta
\end{equation}
for some self-dual $2$-form $\eta\in \Gamma (\Lambda^+)$.
If the harmonic part $\eta_H$ of $\eta$ satisfies 
\begin{equation}
\label{no-harm}
\eta_H \neq 2\pi [c_1(L)]^+,
\end{equation}
where $[c_1(L)]^+$ is the image of $c_1(L)$ under the projection $H^2(M, \RR)\to {\mathcal H}^+_{g}$
defined by \eqref{deco-harm}, then any solution of \eqref{drc} and \eqref{ptsd} is  irreducible, in the
sense that $\Phi\not\equiv 0$. The Smale-Sard theorem  then allows one to show that, 
 for a set of $\eta$ of the second Baire category,
the moduli space 
$$ \mathfrak{M}_\mathfrak{c} (g, \eta ) = \{ \mbox{solutions of  \eqref{drc} and \eqref{ptsd}}\}/ \mathscr{G}$$
is a compact manifold of 
dimension\footnote{When the dimension predicted by \eqref{mdim} is negative, 
this  statement is {\em defined} to mean    that $\mathfrak{M}_\mathfrak{c} (g, \eta ) = \varnothing$
for generic $\eta$.}  
\begin{equation}
\label{mdim}
\dim \mathfrak{M}_\mathfrak{c} (g, \eta ) = \frac{c_1^2(L) - (2\chi + 3\tau )(M)}{4},
\end{equation}
where 
$\chi (M) = (2-2b_ 1+ b_2)(M)$  and $\tau (M) = (b_+-b_-)(M)$  respectively denote the Euler characteristic 
and signature of our smooth compact oriented connected $4$-manifold, and where $c_1^2(L) := c_1(L)\bullet c_1(L)$.

To prove this claim, one  imposes the 
harmless gauge-fixing condition 
 \begin{equation}
\label{gauge} 
d^*(\theta - \theta_0) =0 , 
\end{equation}
relative to some   an arbitrarily chosen reference connection $\theta_0$  on $L$, and  notices that this simply   cuts 
 down the action of the infinite-dimensional gauge group $\mathscr{G}$ to  that of the $1$-dimensional Lie group
$$\mathscr{G}_0 = \{ \mbox{harmonic  maps } {\zap f} : M\to S^1 \}= S^1 \rtimes H^1 (M, \ZZ ).$$
This reduces the problem to understanding the fibers of  the ``monopole map''  
\begin{eqnarray}
L^2_{k} ({\mathbb V}_+)  \oplus  L^2_{k} (\Lambda^1)
&{\longrightarrow}& 
L^2_{k-1}({\mathbb V}_-) \oplus L^2_{k-1}(\Lambda^+)  \oplus L^2_{k-1}/\RR \label{raw}\\
(\Phi, ~\vartheta )\qquad  &\longmapsto& (D_{\theta_0 + i\vartheta} \Phi , ~iF^+_{\theta_0} - d^+\vartheta + \sigma (\Phi ) , ~d^*\theta )
\nonumber
\end{eqnarray}
for any sufficiently large $k$.  By ellipticity, this is a Fredholm map, and the the index theorem, applied to the linearization, 
tells us that the  Fredholm index of \eqref{raw} equals the right-hand side 
of  \eqref{mdim} plus one, where the $+1$ arises from our having modded out by the constant functions $\RR$
in the codomain. The  moduli space $\mathfrak{M}_\mathfrak{c} (g, \eta )$ is then the fiber  over 
 $(0,\eta , 0)$, modulo the action of  $\mathscr{G}_0$. 
Because the linearization of $\dir_\theta \oplus d^*$ at an irreducible solution  always maps surjectively  onto 
$L^2_{k-1}({\mathbb V}_-)   \oplus L^2_{k-1}/\RR$, 
 the Smale-Sard theorem then tells us that  $(0,\eta, 0)$ is a regular value 
for  generic $\eta$,
and  the corresponding fiber is therefore   a smooth manifold by the implicit function theorem. 
On the other hand, 
compactness follows from  the generalization 
\begin{equation}
\label{pwbk}
 0=	2\Delta |\Phi|^2 + 4|\nabla_{\theta}\Phi|^2 +s|\Phi|^2 + |\Phi|^4 - 8 \langle \eta , \sigma (\Phi )\rangle
\end{equation}
of \eqref{wnbk} arising from \eqref{drc} and \eqref{ptsd}, because this Weitzenb\"ock formula
implies that any irreduducible solution satisfies the  point-wise  bound 
\begin{equation}
\label{squeeze}
 |\Phi |^2 \leq \max (2\sqrt{2}|\eta | - s)
\end{equation}
everywhere. Since the action of $\mathscr{G}_0$ also allows us to  assume that the harmonic part of $\vartheta$ is 
bounded,   boot-strapping then  shows that 
$(\Phi , \vartheta)$ belongs to a bounded subset of $L^2_{k+1} ({\mathbb V}_+)  \oplus  L^2_{k+1} (\Lambda^1)$
for any large $k$, and  the Rellich theorem then says that  its image in $L^2_{k} ({\mathbb V}_+)  \oplus  
L^2_{k} (\Lambda^1)$ is necessarily compact. Since \eqref{no-harm} again  implies that every solution is irreducible, 
 $\mathscr{G}_0$ acts
freely and properly, and  $\mathfrak{M}_\mathfrak{c} (g, \eta )$
is therefore  a smooth compact manifold, with  dimension given by \eqref{mdim},  for ``most'' choices of $\eta$.  

To define Witten's invariant, we now restrict ourselves to the case when the ``expected dimension'' \eqref{mdim} 
of the moduli space is {\sf zero}. Now, remarkably enough, the right-hand side of \eqref{mdim} 
 equals $c_2 (\mathbb{V}_+) =e(\mathbb{V}_+)$ for any spin$^c$ structure, so  this
 expected dimension  vanishes  if and only if 
 the spin$^c$ structure $\mathfrak{c}$ is the one  determined   by some  orientation-compatible 
almost-complex structure 
$J$ defined on all of $M$. If $b_+(M) \geq 2$,  any two regular-value choices of $\eta$ satisfying  \eqref{no-harm}
can  be joined by a smooth path $\eta (t)$ satisfying \eqref{no-harm} for all $t$, and such a path 
 then has an arbitrarily  small deformation that is transverse to the monopole map \eqref{raw}. Taking the inverse image
and modding out by $\mathscr{G}_0$ then gives a $1$-dimensional cobordism between $0$-dimensional 
moduli spaces associated with our two points. One may therefore define two invariants that are 
infinite-dimensional analogues of the un-oriented and oriented {degrees of a proper map}. The first  of these is thus  defined  \cite{KM} 
by just setting $n_\mathfrak{c}(M)\in \ZZ_2$ equal to $\# \mathfrak{M}_\mathfrak{c} (g, \eta ) \bmod 2$ when $(0,\eta, 0)$
 a regular value of \eqref{raw}.
The more sophisticated second version \cite{morgan} depends on first defining 
 a consistent  orientation of the moduli spaces, and then defines 
$\mathbf{SW}_\mathfrak{c}(M)\in \ZZ$ to be signed count of the points of the discrete set $\mathfrak{M}_\mathfrak{c} (g, \eta )$.
In either case, the resulting invariant is actually  metric-independent, because one can more generally 
construct cobordisms of the moduli spaces by  considering paths
$(g(t), \eta (t))$ where  the metric also varies. When the invariant  $n_\mathfrak{c}(M)$ or $\mathbf{SW}_\mathfrak{c}(M)$ 
 is non-zero, it then  follows that the Seiberg-Witten equations (\ref{drc}--\ref{sd}) must have a solution, relative to the
given spin$^c$ structure $\mathfrak{c}$,  for any metric $g$. 
Indeed, if there were no solution for a metric $g$ satisfying $[c_1(L)]^+\neq 0$,
then 
$\eta =0$ would satisfy \eqref{no-harm}, and the absence of solutions would then make $(0,0,0)$ 
  a  {\em regular value} of the monopole map \eqref{raw}; thus, the count of solutions would then 
say  that  $n_\mathfrak{c}(M)$ and $\mathbf{SW}_\mathfrak{c}(M)$  both vanished, thereby  contradicting our hypothesis. 
On the other hand, when  $[c_1(L)]^+= 0$ with respect to a given  metric $g$, 
we may  at least produce a {\em reducible} solution of the equations
by setting $\Phi\equiv 0$ and then  choosing $\theta$ so as to make  the $2$-form $F_\theta$ is harmonic; thus, the assertion 
remains  true in this ``bad'' case, albeit for trivial reasons.

The above discussion also    works reasonably well when $b_+(M)=1$, but  one must remember to  pay
 careful attention to a key additional  subtlety. 
 Indeed,  when $b_+(M) = 1$, the vector space 
 $\mathcal{H}^+_g$ is $1$-dimensional,  so   removing a point from $\mathcal{H}^+_g$ therefore disconnects 
it 
into  two open rays. 
Consequently, 
the self-dual $2$-forms $\eta$ satisfying \eqref{no-harm}
then fall into precisely  two connected components. 
Furthermore,  for each spin$^c$ structure $\mathfrak{c}$, 
the set of pairs $(g, \eta)$, where $g$ is a Riemannian metric, and where $\eta$ is a self-dual $2$-form
satisfying \eqref{no-harm} with respect  to $g$,  consists of   exactly two connected 
 components, 
$\sphericalangle^+$ and $\sphericalangle^-$,  called {\sf chambers}. 
Applying the previous discussion to each chamber  then produces  two distinct invariants
$n_{\mathfrak{c}} (M, \sphericalangle^\pm)\in \ZZ_2$, and two distinct invariants $\mathbf{SW}_\mathfrak{c} (M, \sphericalangle^\pm)\in \ZZ$; these are typically different, and an expression for their  difference  is then called a {\sf wall-crossing formula}. 
 As long as $[c_1(L)]^+\neq 0$ for  a given  metric $g$ and spin$^c$ structure $\mathfrak{c}$, the same  arguments
used when $b_+(M)\geq 2$  will then guarantee the existence of an irreducible solution of the ``unperturbed'' Seiberg-Witten equations 
(\ref{drc}--\ref{sd}) provided that 
 $n_\mathfrak{c}\neq 0$ or  $\mathbf{SW}_\mathfrak{c}\neq 0$  for the  chamber containing $\eta=0$.

 \begin{figure}[htb]
  \centerline{
 \includegraphics[scale=.9]{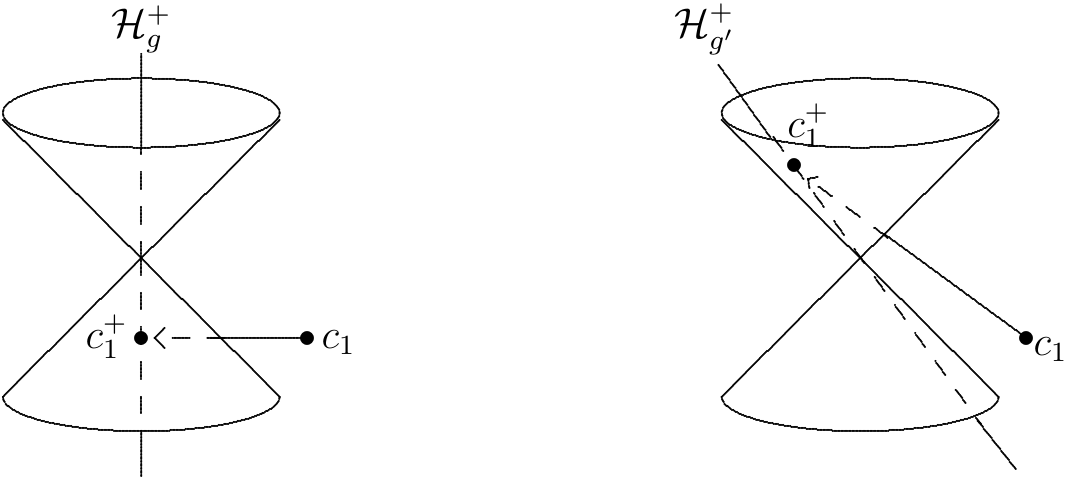}
 }
\caption{\label{cones}
When $b_+(M)=1$, the Seiberg-Witten invariant depends on a choice of {\em chamber}, and 
whether the invariant forces the  Seiberg-Witten equations (\ref{drc}--\ref{sd}) to admit an irreducible  solution 
hinges on whether
the self-dual-harmonic projection  $c_1^+= [c_1(L)]^+\in H^2(M, \RR)$ of  $c_1=c_1(L)$ is past-pointing or future-pointing. 
If $c_1^2(L) < 0$, the answer to this question  genuinely  depends on the metric.  
}
\end{figure}

To clarify this point, notice that if $b_+(M) =1$ and $b_-(M) \neq 0$, then $(H^2(M, \RR), \bullet )$ is essentially  
 a copy of $b_2(M)$-dimensional Minkowski space. The set of ``timelike'' cohomology classes $\alpha\in H^2(M, \RR)$ with
 $\alpha^2:= \alpha \bullet \alpha > 0$ is thus an open double cone consisting  of  two connected components, or {\em nappes}. 
 The choice of a ``time orientation'' for $(H^2(M, \RR), \bullet )$ then amounts to labeling one of these nappes, 
 henceforth denoted by  $\mathscr{C}^+$,  
 as the set of ``future-pointing'' time-like vectors, while  declaring that  the remaining  nappe  $\mathscr{C}^-$ 
consists of ``past-pointing'' time-like vectors.
The impact of the chamber-dependent  invariants on the unperturbed Seiberg-Witten equations (\ref{drc}--\ref{sd})
is  then entirely  governed by   the image
$[c_1(L)]^+\in \mathcal{H}_g$ of the first Chern class 
under Minkowski-space-orthogonal projection   to  $\mathcal{H}^+_g\subset H^2(M,\RR)$. 
Since $\mathcal{H}_g^+-\{ 0\}\subset \mathscr{C}^+ \cup \mathscr{C}^-$ for any Riemannian metric $g$, 
 the ``no reducible solutions'' condition $[c_1(L)]^+\neq 0$ implies 
  that either $[c_1(L)]^+\in \mathscr{C}^-$ or $[c_1(L)]^+\in \mathscr{C}^+$.  In the first case, the self-dual $2$-form
$\eta =0$ belongs to the chamber $\sphericalangle^+$ 
containing the open-ray component of $\mathcal{H}_g^+-\{ 0\}$ that terminates in $\mathscr{C}^+$,
while in the second case $\eta =0$ belongs to the chamber $\sphericalangle^-$
containing the open-ray component of $\mathcal{H}_g^+-\{ 0\}$ that terminates
 in $\mathscr{C}^-$. Thus, in order to use the invariant $n_\mathfrak{c}(M, \sphericalangle^\pm)$ or  $\mathbf{SW}_\mathfrak{c}(M, \sphericalangle^\pm)$ to predict the the existence of solutions to
 (\ref{drc}--\ref{sd}), one must simply  keep track of whether $[c_1(L)]^+$ is future-pointing or past-pointing for a given metric $g$.
 If $c_1^2(L) \geq 0$ and $c_1(L)$ is not a torsion class,    the answer to this question is  actually independent $g$. 
 However, when $c_1^2(L) <  0$, the answer  becomes dependent on the metric, as indicated by {\sf Figure \ref{cones}}. 
 Fortunately, this technical inconvenience 
  can  often be overcome    \cite{FM,lno} by carefully playing several spin$^c$ structures off against one another.

After  outlining  the  the definition of the Seiberg-Witten invariant in the $b_+(M)\geq 2$ case, 
Witten then  went on to argue that  the invariant is non-zero   when   $M$ is also the underlying smooth oriented $4$-manifold 
of a compact-complex surface  of K\"ahler type, equipped with the spin$^c$ structure determined by the 
complex structure. Recall \cite{bpv,nick,siu} that a compact complex surface $(M^4, J)$
admits compatible K\"ahler metrics iff its first Betti number $b_1(M)$ is even, and that this happens
iff $b_+(M)$ is odd. 
To gain some insight into Witten's claim regarding  K\"ahler manifolds, let us first prove a  general technical result \cite{lno} 
that indicates the sense 
 in which
the Seiberg-Witten equations generalize  key aspects of K\"ahler geometry to a general Riemannian 
setting:

\begin{proposition}\label{best}
Let $(M,g)$ be a smooth compact oriented Riemannian $4$-manifold,
let $\mathfrak c$ be a spin$^c$ structure on $M$, and let 
$c_1^+= [c_1(L)]^+$ denote the self-dual part of the harmonic $2$-form 
representing the first Chern class $c_1(L)$ of $\mathfrak c$. 
If   there is a  solution of the Seiberg-Witten equations (\ref{drc}--\ref{sd}) 
on $M$ for $g$ and  ${\mathfrak c}$, then the scalar curvature 
$s_{g}$ of $g$ satisfies 
$$
\int_{M}s_{g}^{2}d\mu_{g} \geq 32\pi^{2} [c_{1}^{+}]^{2} .
$$
Moreover, when $[c_1^+]\neq 0$,  equality can  occur only if $g$
is a K\"ahler metric of constant {negative} scalar curvature that is compatible with 
a complex structure $J$ such that $c_1(M,J) = c_1(L)$. 
\end{proposition}
\begin{proof}
Integrating the Weitzenb\"ock formula (\ref{wnbk}), we have
$$0= \int [ 4|\nabla \Phi |^2 + s|\Phi|^2 + |\Phi|^4 ] d\mu , $$
and it follows that 
$$\int (-s) |\Phi|^2 d\mu \geq \int |\Phi|^4 d\mu .$$
Applying the Cauchy-Schwarz inequality to the
left-hand side therefore yields  
$$
\left(\int s^2 d\mu  \right)^{1/2}\left(\int |\Phi |^4 d\mu  \right)^{1/2} \geq \int |\Phi |^4 d\mu .
$$
We therefore have  
$$
\int s^2 d\mu \geq \int |\Phi |^4 d\mu  = 8 \int |F_{\theta}^{+}|^2 d\mu ,
$$ 
and   the inequality is strict unless  $\nabla_\theta \Phi \equiv 0$ and $s$ is a non-positive constant. 
However, $F_{\theta}^{+}-2\pi c_1^+$  is an exact form plus a co-exact form, and 
so is $L^2$-orthogonal to the harmonic  forms. This gives us
the inequality 
$$\int |F_{\theta}^{+}|^2 d\mu \geq 4\pi^2 \int |c_1^+|^2 d\mu = 4\pi^2 \int c_1^+\wedge c_1^+,$$
and the last expression may be re-interpreted as the intersection pairing 
$[c_1^+]^2$ 
of 
the de Rham class of $c_1^+$ with itself. This gives us the desired
inequality
$$\int s^2 d\mu \geq 32\pi^2 [c_1^+]^2,$$
and, when the right-hand side is non-zero, equality can only happen if  
$\sigma (\Phi )$  is parallel 
and $g$  has constant negative scalar curvature.  
\end{proof}

Conversely, suppose that $(M,g,J)$ is a compact K\"ahler surface of 
constant negative scalar curvature; and let  $\omega = g(J\cdot , \cdot )$ denote the corresponding K\"ahler form. 
For the  spin$^c$ structure
$\mathfrak{c}$ determined by $J$, one then has 
$$\mathbb{V}_+= \Lambda^{0,0} \oplus \Lambda^{0,2}, \quad \mathbb{V}_- =\Lambda^{0,1},$$
and the spin$^c$ Dirac
operator $\dir_\theta$ determined by the Chern connection $\theta$ on the 
anti-canonical line bundle $K^{-1}=L=\Lambda^{0,2}$ is just $$\sqrt{2} (\bar{\partial} + \bar{\partial}^*): \Lambda^{0,0} \oplus \Lambda^{0,2}\to \Lambda^{0,1}.$$
Moreover, the bundle of real self-dual $2$-forms of $g$ is just 
$$\Lambda^+ = \RR \, \omega \oplus \Re e (\Lambda^{0,2})$$
while   $\sigma : \mathbb{V}_+\to \Lambda^+$ is exactly given by 
$$\sigma (f, \phi) =  (|f|^2 - |\phi|^2 ) \frac{\omega}{4} + \Im m (\bar{f}\phi).$$
Because the curvature of $(K^{-1} , \theta )$ is exactly $-i\rho$, where the Ricci form 
$\rho$ has self-dual part $\rho^+ = \frac{s}{4}\omega$, 
it therefore follows that $\Phi = (\sqrt{-s} , 0)$ and the Chern connection $\theta$ together solve the Seiberg-Witten equations
(\ref{drc}--\ref{sd}). Moreover, because the Ricci form $\rho$ of any constant-scalar-curvature  K\"ahler metric $g$
is harmonic, the harmonic representative of $c_1= c_1(L)$ is $\frac{\rho}{2\pi}$  in the present setting, and 
its self-dual piece is therefore  $c_1^+ = \frac{s}{8\pi} \omega$. It follows that any Seiberg-Witten solution
exactly saturates the inequalities in the proof of Proposition \ref{best}, so that $\Phi$ must be parallel with respect 
to $\theta$. This then implies that any  solution is gauge-equivalant to the explicit one we have displayed. 
Moreover, direct calculation \cite{lmo} shows that the derivative of the monopole map \eqref{raw}  is surjective at our explicit solution. 
This shows that  $\mathbf{SW}_\mathfrak{c}(M)=\pm 1$ and $n_\mathfrak{c}(M)\neq 0$ if
$b_+(M) > 1$, while $\mathbf{SW}_\mathfrak{c}(M,\sphericalangle^+)=\pm 1$ and $n_\mathfrak{c}(M,\sphericalangle^+)\neq 0$ 
for the  appropriate chamber $\sphericalangle^+$ if $b_+(M)=1$.

As a consequence,  we therefore obtain the  following result \cite{lmo,lno}:

\begin{theorem} \label{optimal} 
Let $(M^4,J)$ be a compact complex surface that admits a compatible K\"ahler-Einstein metric $g$
with  scalar curvature $s<0$. Then
 $$\mathcal{I}_s (M) = 32\pi^2 \, c_1^2(M,J) >  0$$ and 
 $$\mathscr{Y}(M)= - 4\pi \sqrt{2\, c_1^2(M,J)} <0$$ are both achieved by $g$.
 Moreover, 
if   $g^\prime$ is a Riemannian metric on $M$ that achieves  one of these critical values for the relevant Riemannian 
functional, then $g^\prime$ 
is also K\"ahler-Einstein, and is compatible with an  integrable complex structure $J^\prime$ such that 
$c_1(M, J^\prime ) = c_1 (M, J) \in H^2(M, \ZZ)$. 
\end{theorem}
\begin{proof} Since $(M,J)$ is of K\"ahler type, $b_+ = 1 + h^{2,0}(M,J)\geq 1$,
and the above discussion of Seiberg-Witten invariants therefore applies. 
On the other hand, the K\"ahler-Einstein metric $g$ has Ricci-form $\rho = \frac{s}{4}\omega$, and hence 
$$c_1^2 (M, J) = \frac{1}{4\pi^2} [\rho ]^2 =  \frac{s^2}{64 \pi^2} [\omega ]^2 > 0.$$ Thus, for the 
spin$^c$ structure determined by $J$, the mod-$2$ count  of solutions of (\ref{drc}--\ref{sd}) 
mod gauge is metric independent, even if $b_+ = 1$. However, since $g$ is K\"ahler, with constant negative
scalar curvature, the   count of solutions is $1\bmod 2$ for $g$ and the spin$^c$ structure determined by $J$, so
 there must therefore  be a solution for any 
other metric. Proposition \ref{best} therefore tells us that 
\begin{equation}
\label{amazing}
\int s^2d\mu \geq 32\pi^2 [c_1^+]^2 \geq 32\pi^2 c_1^2(M, J)
\end{equation}
for any Riemannian metric $g^\prime$ on $M$. Thus
$$\mathcal{I}_s (M) = \inf_{g^\prime}  \mathfrak{S} (g^{\prime}) = 32\pi^2 c_1^2(M, J)$$
and the infimum is moreover attained by $g$. Since this means, in particular, that $\mathcal{I}_s (M) > 0$,
we therefore have
$$\mathscr{Y}(M) = - \sqrt{\mathcal{I}_s (M)}= -4\pi \sqrt{ 2\, c_1^2 (M,J)} < 0.$$
Moreover, if $g^\prime$ saturates \eqref{amazing}, then $g^\prime$ is a K\"ahler metric of constant negative
scalar curvature by Proposition \ref{best}; and because $[c_1^+]^2= c_1^2(M,J)$, 
the Ricci form of $g^\prime$, which is the 
harmonic representative of $2\pi c_1$ with respect to  this metric, must also  
be self-dual,  thus  implying  that $g^\prime$ is    K\"ahler-Einstein. \end{proof}

Theorem \ref{option} is now an immediate consequence. Indeed, consider the complex  surfaces
$M_\ell  \subset \CP_3= \{ [ z_0 : z_1 : z_2 : z_3 ]~|~z_j \in \CC \} $ defined by 
$$z_0^\ell + z_1^\ell + z_2^\ell + z_3^\ell =0.$$
Each is  the zero set of a  section of a positive line bundle, namely  $\mathcal{O}(\ell)$,  that is transverse to the zero section, 
so each is 
simply connected by  Lefschetz's theorem on  hyperplane sections. On the other hand, the canonical line bundle
of $M_\ell$ is exactly $\mathcal{O} (\ell - 4)$ by the adjunction formula, and so is ample if $\ell > 4$.
The Aubin-Yau theorem \cite{aubin,yau} therefore implies that for each $\ell \geq 5$, the complex surface  $M_\ell$
admits a negative-scalar-curvature
compatible K\"ahler-Einstein metric,
and Theorem \ref{optimal} therefore tells us that 
$$\mathcal{I}_s (M_\ell) = 32\pi^2 c_1^2(M_\ell ) = 32\pi^2 \ell (\ell - 4)^2, \qquad \forall \ell \geq 5,$$
thus providing a sequence of $4$-manifolds for which $\mathcal{I}_s\to + \infty$, as claimed. 
Moreover, for the sequence $4$-manifolds we have just displayed, 
 the infimum $\mathcal{I}_s = \inf \mathfrak{S}$ is achieved in each instance by an 
Einstein metric (that, for these purposes,  just happens to be K\"ahler).

An important  consequence of this discussion is that the Yamabe invariant can often be used to  distinguish between 
different smooth structures on a fixed topological $4$-manifold. For example, consider the 
underlying smooth compact oriented $4$-manifolds represented by the above complex surfaces $M_\ell$.
These manifolds are simply connected and have
$$b_+(M_\ell ) = \frac{(\ell -1) (\ell -2) (\ell -3) }{3} +  1 , \quad b_-(M_\ell) = \frac{(\ell + 2)\ell (\ell -2)}{3} + b_+(M) .$$
However,  a deep theorem due to  Freedman \cite{freedman}  implies that two compact connected oriented {\em smooth}  
$4$-manifolds are orientedly {\em homeomorphic} if and only if they have the same invariants $b_+$ and $b_-$, and  either 
both  are spin, or both  are non-spin. For  $\ell =2k+1$ odd, with $k\geq 2$, this for instance tells us  that
the negative-Yamabe-invariant non-spin $4$-manifold $M_{2k+1}$  is 
 homeomorphic to the positive-Yamabe-invariant connected sum 
${\zap m} \CP_2 \# {\zap n}\overline{\CP}_2$,  where
  ${\zap m}= 1+ \frac{4}{3}k(k-1)(2k-1)$,  and ${\zap n}= \frac{2}{3}k(8k^2+1)$. Similarly, when   $\ell=2k$ is even and $k\geq 3$, 
the negative-Yamabe-invariant spin manifold  $M_{2k}$ 
 is  homeomorphic to the zero-Yamabe-invariant\footnote{The fact that these connected sums have $\mathscr{Y}\leq 0$ follows from the fact that they  are spin 
and have non-zero $\widehat{A}=-\tau/8$. The fact that they have $\mathscr{Y}\geq 0$ follows from Petean's surgery lemma \cite{jp3}, together with the fact that $\mathsf{K3}$ admits a  Ricci-flat  (and hence scalar flat) metric by Yau's solution of the Calabi conjecture \cite{yau,yauma}.}
 connected sum ${\zap p} \mathsf{K3} \# {\zap q} (S^2\times S^2)$, where 
${\zap p} = {k+1\choose 3}$,  ${\zap q}=  \frac{1}{6}(k-2)(13k^2-22k+3)$,  and where 
$\mathsf{K3} := M_4$ is one of the essential building blocks of $4$-dimensional topology. 

While the role of the scalar curvature in this story was a stunning consequence of the advent 
of Seiberg-Witten theory,  it is worth remembering that    it was previously  known that 
  the complex surfaces $M_\ell$, $\ell \geq 4$, were certainly not diffeomorphic to the connect sums considered above;
  this had been  proved by Donaldson  \cite{don}, using  new  polynomial  invariants that  he had defined using 
moduli spaces of Yang-Mills instantons.  It should also be noted Donaldson's Yang-Mills-based     thesis result \cite{donaldson}
 is actually needed    to streamline the smooth-manifold case  of 
 Freedman's theorem into the user-friendly   statement given above.  Witten's ostensible justification 
 for introducing the Seiberg-Witten equations in the first place was a physics argument indicating that 
 his new theory should encode exactly same information as the Donaldson polynomials. 
 While  partial results \cite{feele} and large amounts of practical experience strongly  indicate that  Witten's claim  is actually  true, 
the intuition linking the two theories continues to largely elude the mathematical community. 
 In particular, 
 the Yang-Mills equations, unlike the Seiberg-Witten equations, are {\em conformally invariant},
 and so, at least locally, are utterly insensitive to the scalar curvature!

Now, ``most'' compact complex surfaces of K\"ahler type  are   deformation-equivalent to 
surfaces which admit compatible K\"ahler metrics of constant negative scalar curvature \cite{arpa1,yujen},
so citing this fact  while  appealing to the converse of Proposition \ref{best}  proved  in the discussion above  provides a 
correct-but-inefficient way of showing  
that   complex surfaces of K\"ahler type  typically carry non-trivial Seiberg-Witten invariants. However, 
one can easily prove more by instead just considering well-chosen  perturbations of the Seiberg-Witten equations
for an arbitrary K\"ahler metric. Indeed, if $t$ is any positive constant,
then $\Phi = (t,0)$ and the Chern connection $\theta$ together solve 
\eqref{drc} and \eqref{ptsd} for the perturbation 
 $\eta = \frac{t^2 + s}{4}\omega$, and the linearization of the monopole map is moreover 
 surjective at this solution; 
careful inspection of   \eqref{pwbk} then shows \cite{spccs} that, up to gauge equivalence, this is
the unique solution of \eqref{drc} and \eqref{ptsd}  for this specific $\eta$. When $b_+(M) > 1$,
this shows that  $\mathbf{SW}_\mathfrak{c} (M)=\pm 1$  on any K\"ahler-type complex surface, and hence that 
$\mathfrak{n}_\mathfrak{c} (M)= 1\bmod 2$, 
where  $\mathfrak{c}$ is or the the spin$^c$ structure  determined by the complex stucture 
by $J$;  when $b_+(M) = 1$, one similarly concludes that  $\mathbf{SW}_\mathfrak{c} (M,\sphericalangle^+)=\pm 1$,
and hence that $\mathfrak{n}_\mathfrak{c} (M,\sphericalangle^+)= 1 \bmod 2$, 
for the chamber $\sphericalangle^+$ containing large positive multiples of the K\"ahler form $\omega$. In particular, 
when $b_+(M) > 1$, there are solutions of the  unperturbed equations (\ref{drc}--\ref{sd}) for any metric on $M$ 
for this specific spin$^c$ structure; when $b_+(M)=1$, one   instead gets solutions of the unperturbed equations
whenever the self-dual projection $c_1^+$ of $c_1(M,J)$ belongs to the nappe $\mathcal{C}^-\subset H^2(M, \RR)$ 
that {\sf does not} contain the K\"ahler class $[\omega ]$ of a reference K\"ahler metric on $(M,J)$. 

One of the fundamental operations of complex surface theory \cite{bpv,GH} is the {\em blow-up} operation, 
which replaces a point of a complex surface $N$ with a $\CP_1$ of normal bundle $\mathcal{O}(-1)$; this then produces 
a new complex surface $M$ that is diffeomorphic to  $N \# \overline{\CP}_2$, where 
$ \overline{\CP}_2$ is the oriented manifold obtained from $\CP_2$ by reversing 
its orientation. Conversely, any complex surface $M$ containing a $\CP_1$ of normal bundle $\mathcal{O}(-1)$
can be ``blown down'' to produce a complex surface $N$ such that $M$ becomes  its blow-up. 
This operation can in principle be iterated, but the process must terminate after finitely many steps, 
because each blow-down decreases $b_2$ by $1$. When  a complex surface
$X$ cannot be blown-down, it is called {\em minimal}, and the upshot is that any
complex surface $M$ can be obtained from a minimal complex surface\footnote{Unfortunately, this means that 
 the term {\em minimal surface} is  widely used by algebraic geometers 
 to mean something that has nothing whatsoever to do with soap bubbles!} 
$X$ by blowing up finitely many times. In this situation,  one then says that $X$ is a {\em minimal model} of $M$. 

With this said, we are  now ready to introduce one of the most important complex-analytic 
 invariants of a compact complex surface, namely its {\em Kodaira dimension}  \cite{bpv,GH}.
This is defined in terms of positive powers of the canonical line bundle $K:=\Lambda^{2,0}$, 
and is given by 
$$\kod (M, J) = \limsup_{j \to +\infty}  \frac{\log \dim H^0 (M,K^{\otimes j}))}{\log j}.$$
The only possible values of this invariant are $-\infty$, $0$, $1$ and $2$, because 
the Kodaira dimension is actually  just the 
 largest complex dimension\footnote{In this context,  we conventionally set $\dim \varnothing :=-\infty$.} 
of the image of $M\dasharrow \mathbb{P} [ H^0 (M,K^{\otimes j}))]^*$ under all the various  ``pluricanonical''  maps 
associated with the  line bundles $K^{\otimes j}$, $j \in \ZZ^+$.
Blowing up or down always  leaves the Kodaira dimension unchanged, and the minimal model of a complex surface is moreover 
{\em unique} whenever  $\kod \neq -\infty$.
As an illustration, for  the surfaces $M_\ell \subset  \CP_3$ discussed above, it is not hard to show that 
$$
\kod (M_\ell )  = \left\{
\begin{array}{ll} -\infty  &\mbox{if  }  \ell \leq 3\\
~ 0 &\mbox{if } \ell = 4  \\
 ~  2 &\mbox{if } \ell \geq 5.
\end{array}\right.
$$ 
This sequence of examples begins to explain why   complex surfaces with   $\kod = 2$ are said to be 
of {\em general type}. Notice that, among these specific examples,  the ones of general type are exactly those for which
seen that the Yamabe invariant  is also negative. This is simply  the first glimmer   of  the following   general pattern \cite{lky}:

\begin{theorem} \label{bacalao} 
Let $M$ be the  smooth 4-manifold underlying a
 compact complex  surface $(M^4,J)$ with $b_1(M)$ even.  
Then 
\begin{eqnarray*}
\mathscr{Y}(M) > 0 &\Longleftrightarrow& \kod (M,J) = -\infty \\
\mathscr{Y}(M) = 0 &\Longleftrightarrow& \kod (M,J) = 0 \mbox{ or } 1 \\
  \mathscr{Y}(M)<  0 &\Longleftrightarrow& \kod (M,J) = 2.
  \end{eqnarray*}
\end{theorem}
Here the assumption that $b_1(M)$ is even is equivalent \cite{bpv,nick,siu} to $(M,J)$ admitting a compatible 
K\"ahler metric; thus,   complex surfaces with $b_1$  even are said to be of {\sf K\"ahler type}.
In \S \ref{yzero} below will we see that 
part, but only part,  of this pattern 
still holds true for complex surfaces with $b_1$ odd.
In particular, Theorem \ref{ellipsis} below and  previously-known results together imply  the 
 following result, which does not depend on the parity of $b_1$:

\begin{theorem} 
\label{parity}
Let  $(M,J)$ be a 
 compact complex  surface with $\kod(M,J) \geq 0$, and let 
 $(X, \check{J})$ be its minimal model. Then 
 $$\mathscr{Y}(M) = \mathscr{Y}(X).$$
\end{theorem}

When $\kod (M,J) = 2$,  this takes the following quantitative  form \cite{lno}: 

\begin{theorem} 
\label{genotype} 
Let $(M^4,J)$ be a compact complex surface of general type, and let  $(X, \check{J})$ be its minimal model. 
Then 
$$\mathcal{I}_s (M) = \mathcal{I}_s (X) = 32\pi^2 c_1^2 (X) > 0$$
and 
$$\mathscr{Y}(M) = \mathscr{Y}(X) = -4\pi \sqrt{2 c_1^2(X)}  < 0.$$
\end{theorem}

The proof first uses  Seiberg-Witten theory to show that  $\mathcal{I}_s (M) \geq 32\pi^2 c_1^2 (X)$; 
this is made possible  by the fact that  every surface of general  type is  K\"ahler (and indeed projective algebraic), 
for reasons  related to  the fact that the minimal model $X$ satisfies $c_1^2(X) > 0$. If 
  $b_+(M) = b_+(X)> 1$, 
the real Chern class 
$$c_1^\RR(M, J) \in H^2 (M,\ZZ)/\mbox{torsion}\subset H^2(M, \RR) $$
 is a so-called {\sf basic class}, 
meaning that it is the Chern class of a 
spin$^c$ structure $\mathfrak{c}$  with  non-zero Seiberg-Witten invariant. If $M=X$ is minimal, the desired lower bound
follows from Proposition \ref{best}.
Otherwise,   $M\approx X\# k\overline{\CP}_2$  has 
$H^2 (M) = H^2(X) \oplus  H^2(k \overline{\CP}_2)$, and $M$ has a self-diffeomorphism $\varphi : M \to M$ 
that  acts by $+1$ on  $H^2(X)$, but as $-1$ on $H^2(k \overline{\CP}_2)$. Since $\varphi$ sends
basic classes to basic classes, this means  that $c_1(X)$ is the average of two 
basic classes $\mathsf{a}_1 = c_1(M, J)$ and $\mathsf{a}_2 = c_1(M, \varphi_*J)$, 
and its projection $c_1(X)^+$ to $\mathcal{H}^+_g$ is therefore
the average of the orthogonal  projections  $\mathsf{a}_1^+, \mathsf{a}_2^+\in \mathcal{H}^+_g$ of these classes
with respect to the intersection form $\bullet$. In particular, 
we must have 
$$ \max ([\mathsf{a}_1^+]^2, [\mathsf{a}_2^+]^2)\geq [c_1(X)^+]^2 \geq c_1^2(X)$$ 
because  $\bullet$ is positive-definite on $\mathcal{H}^+_g$.   Proposition \ref{best} therefore 
implies
$$\mathfrak{S} (g) = \int_M s_g^2 d\mu_g \geq 32\pi^2 \max ([\mathsf{a}_1^+]^2, [\mathsf{a}_2^+]^2)  \geq 32\pi^2 c_1^2(X).$$
When $b_+(M)= b_+(X) =1$, the proof  is similar, but also depends on   the fact that 
$\max ([\mathsf{a}_1^+]^2, [\mathsf{a}_2^+]^2)$ is achieved by a spin$^c$ structure for which 
$\mathsf{a}_j^+$ is past-pointing. 
Either way, we now see that   $\mathcal{I}_s (M)= \inf \mathfrak{S} \geq 32\pi^2c_1^2(X)$. 

To finish the proof, we now just need to construct a sequence of metrics $g_j$ on $M$ for which  $\mathfrak{S} (g_j)$
tends to  this lower bound. This is made possible by the fact that the  {\em pluricanonical 
model} $\check{X}$ of $M$ is a complex orbifold with $c_1< 0$, so  that the proof of the Aubin-Yau
theorem therefore endows it  \cite{tsuj} 
with an orbifold K\"ahler-Einstein metric  satisfying $\int s^2 d\mu = 32\pi^2 c_1^2({\check{X}})$.
However, the natural map  $X\to \check{X}$ relating the minimal and pluricanonical models is 
 a so-called crepant resolution, meaning that $c_1(X)$ is just the pull-back
of $c_1(\check{X})$. Because 
the orbifold singularities of $\check{X}$ are 
all modeled on $\CC^2/\Gamma$ for finite groups $\Gamma\subset \mathbf{SU}(2)$, 
each can be replaced with a  gravitational instanton 
(asymptotically-locally-Euclidean Ricci-flat K\"ahler  manifold) \cite{kron} 
to obtain a Riemannian metric on the minimal model $X$, while
increasing  $\int s^2 d\mu$ by an arbitrarily small amount arising  from the transition region. 
Similarly, to pass from the minimal model $X$ to the given complex surface $M\approx X\# k \overline{\CP}_2$, 
one can use the fact that $\overline{\CP}_2-\{ pt\}$ admits an asymptotically-flat scalar-flat metric called
the {\em Burns metric} \cite{lpa}; gluing in rescaled copies of this model then gives  metrics on $M$ that 
with $\mathfrak{S} < 32\pi^2 c_1^2({X}) + \varepsilon$ for any given $\varepsilon > 0$.

\bigskip

If $M^4$ admits a K\"ahler metric, and if $\mathfrak{c}$ is the spin$^c$ structure determined by the associated 
integrable almost-complex structure $J$, we have already seen that the Seiberg-Witten invariant is non-zero if $b_+ > 1$,
and that the Seiberg-Witten invariant is similarly non-zero in a  suitable chamber  if $b_+ =1$. 
 However, Taubes \cite{taubes} 
 proved a dramatic generalization of  this result:
  
 \begin{theorem}[Taubes]
 Let $(M^4,\omega )$ be a compact symplectic $4$-manifold, and let $\mathfrak{c}$ be the 
 spin$^c$ structure on $M$ determined by an almost-complex structure $J$  compatible 
 with $\omega$. If  $b_+(M) > 1$, then $\mathbf{SW}_\mathfrak{c}(M) = \pm 1$. 
Similarly, if  $b_+(M) = 1$, then $\mathbf{SW}_\mathfrak{c}(M, \sphericalangle^+) = \pm 1$,
where $ \sphericalangle^+$ is the chamber containing large positive multiples of the symplectic form $\omega$. 
 \end{theorem} 
 The strategy of the proof is similar to what has been described in the K\"ahler case, but is technically far more subtle. 
 One chooses a so-called almost-K\"ahler metric $g$ that is related to the symplectic form  by 
 $\omega = g(J \cdot, \cdot )$ for some  almost-complex structure $J$, and then writes down an explicit solution of (\ref{drc}--\ref{ptsd}) 
for  perturbations $\eta$ of the form $t^2 \omega + \eta_0$. When  $t \gg 0$ is sufficiently large, one then uses the 
Weitzenb\"ock formula to show that, up to gauge equivalence,  the explicit solution is actually the only solution. 

In a stunning pair  of sequels \cite{taubes2,taubes3}, Taubes then extended these ideas in ways 
that  completely revolutionized $4$-dimensional symplectic geometry, and gave  the  trail-blazing earlier 
contributions of Gromov \cite{gromsym}
and McDuff \cite{mcrules,dusa}   surprising  new meanings. While McDuff had previously shown that
blowing up and down had natural generalizations to symplectic $4$-manifolds, Taubes now showed that every symplectic
$4$-manifold $M$ with $b_+(M)> 1$ has a unique minimal model $X$ with $c_1^2(X)\geq 0$, and that 
symplectic minimality or non-minimality for such manifolds depends only on diffeomorphism type. 
Taubes' also showed that the anti-canonical class $-c_1$ on any such symplectic manifold is always
represented by one of Gromov's  pseudo-holomorphic curves,  so that every symplectic $4$-manifold with 
$b_+> 1$ in particular  satisfies $c_1\bullet [\omega ] < 0$. When $b_+=1$, versions of Taubes' results still hold, 
 but in this context one must carefully check  whether the Seiberg-Witten invariant is non-zero in chamber needed for any given application. 
 
  Taubes' results immediately  made it  easy to extend part of Theorem \ref{genotype}   to $4$-manifolds that 
 support symplectic structures. The key definition needed to make  this possible was the following \cite{lno}: 
 
 \begin{definition} 
 \label{sympathy} 
 A symplectic $4$-manifold $(M,\omega)$ is said to be of {\sf general type} 
 iff its minimal model $(X, \check{\omega})$ satisfies 
 \begin{itemize}
 \item $c_1^2(X) > 0$; and 
 \item $c_1(X) \bullet [\check{\omega}] < 0.$
 \end{itemize}
 \end{definition}

 \noindent 
 Kodaira classification then implies that a K\"ahler manifold $(M^4, J , \omega)$ 
  is of general type as a complex surface iff it 
 is  of general type as a symplectic $4$-manifold, and  essentially  the same proof as before  
 then yields  \cite{lno}  the following partial generalization
  of Theorem \ref{genotype}:

\begin{theorem} 
\label{phenotype} 
Let $(M^4,\omega)$ be a symplectic manifold  of general type, and let  $(X, \check{\omega})$ be its minimal model. 
Then \begin{equation}
\label{symphony}
\mathcal{I}_s (M) \geq   32\pi^2 c_1^2 (X) > 0 \quad \mbox{and}\quad \mathscr{Y}(M) \leq  -4\pi \sqrt{2 c_1^2(X)}  < 0.
\end{equation}
\end{theorem}

Together with  constructions by Gompf \cite{gompf}, Fintushel-Park-Stern \cite{FPS},  and others, Theorem \ref{phenotype} 
implies that there are many symplectic $4$-manifolds with 
negative Yamabe invariant that do not arise as complex surfaces. On the other hand, Theorem \ref{phenotype} 
merely gives an upper bound for  the Yamabe invariant, rather than actually calculating its actual value, and my 
guess at the time was that this bound would never be sharp in the non-K\"ahler case. However, this turned out to be wrong. 
Indeed, Ioana \c{S}uvaina \cite{ioana} later succeeded in proving that certain  symplectic manifolds arising from Gompf's
rational-blow-down construction  have Yamabe invariants that do exactly saturate  \eqref{symphony}, 
even though they do not admit K\"ahler metrics. In her examples, 
the rational-blow-down operation can be carried out by starting with a K\"ahler-Einstein orbifold, and then gluing
in finite quotients of Gibbons-Hawking gravitational instantons \cite{ioana2}. Whether this phenomenon is 
rare or common seems to be largely open, and would seem to deserve a thorough and systematic   investigation.

It is also worth pointing out  that 
Definition \ref{sympathy}  was eventually  re-discovered by Li \cite{likod}, 
who went on  to provide  a systematic definition of Kodaira dimension for symplectic $4$-manifolds. 
This of course makes it tempting to ask
 whether, with  Li's definitions,  Theorems \ref{bacalao} and/or  \ref{parity} also hold for symplectic $4$-manifolds. 
While this might seem unlikely, the  issue is one with enough  intrinsic  interest to  definitely make it  worth
settling, one way or the other.

\bigskip 

Even though Witten's invariant leads to so many compelling results, our applications of Seiberg-Witten theory to the 
 Yamabe invariant have 
really only depended on being guaranteed a solution of the Seiberg-Witten equation for every 
metric on a given $4$-manifold. As it turns out, this can happen even in contexts where Witten's invariant vanishes
for every spin$^c$ structure. For example, Bauer and Furuta \cite{baufu,bauer2} discovered a generalization of
the $\ZZ_2$-valued Seiberg-Witten invariant $n_\mathfrak{c}$ that implies such an existence 
statement in contexts where the expected dimension \eqref{mdim} of the moduli space is positive. 
In the simply-connected case, their invariant is the element of the equivariant stable cohomotopy groups 
$\pi_{S^1}^{b_+} (\Ind \dir)$ represented, via a finite-dimensional-approximation scheme, 
 by the monopole map \eqref{raw}, thought of as a map between 
one-point compactifications of Hilbert spaces. When this invariant is non-zero, 
one then concludes that there is a solution of (\ref{drc}--\ref{sd}) for every Riemannian metric $g$ on $M$. 
Because this property is useful in and of itself, it is worth codifying, as follows:

\begin{definition}
Let $M$ be a smooth compact oriented $4$-manifold
with $b_{+}\geq 2$. An element $\mathbf{a}\in  H^{2}(M,\ZZ )/
\mbox{\rm torsion}$, $\mathbf{a}\neq 0$,  is  called a {\sf monopole
class} of $M$ iff there is some  spin$^{c}$ structure
$\mathfrak{c}$ 
on $M$ with first Chern class 
$$c_{1}(L)\equiv \mathbf{a} ~~~\bmod \mbox{\rm torsion}$$ for which   the   Seiberg-Witten 
equations (\ref{drc}--\ref{sd})
have a solution for every Riemannian  metric $g$ on $M$. 
\end{definition}

While this is a completely ``soft'' definition\footnote{Here the   hypothesis $b_+\geq 2$  is primarily imposed
  as a matter of taste rather than of necessity, because  
 the related  notion of  a {\em retroactive class} better encapsulates the essential  features of
the  usual  Seiberg-Witten invariant when $b_+=1$. For details, see \cite{lebeta}.}, it   nonetheless has some useful  immediate 
consequences \cite{lebeta}: 

\begin{proposition} 
If $M$ is  a smooth compact oriented manifold with $b_+ \geq 2$, then 
$$\mathfrak{C} = \{ \mbox{monopole classes on } M\} \subset H^2(M, \RR)$$
is a diffeomorphism invariant of $M$. Moreover,  $\mathfrak{C}$
is a {\sf finite} set, and  is invariant under reflections $\mathbf{a} \mapsto - \mathbf{a}$ through the origin.
\end{proposition}

As a consequence, if we define $\hull \mathfrak{C}\subset H^2(M,\RR)$ to be the convex hull of the monopole classes,
then  $\hull \mathfrak{C}$  is  compact, and necessarily contains  $0$ if $\mathfrak{C}\neq \varnothing$. This allows us to 
define  a  non-negative numerical invariant of any smooth compact oriented $4$-manifold that that provides a transparent 
distillation 
of many  key arguments in the subject:

\begin{definition}
Let $M$ be a smooth compact oriented $4$-manifold
with $b_{+}\geq 2$. If $\mathfrak{C}= \varnothing$, set $\beta (M) =0$. Otherwise, 
let 
$$\beta^2 (M) = \max \{ ~\mathsf{a} \bullet \mathsf{a} ~|~ \mathsf{a}\in \hull \mathfrak{C} \} $$
and set $\beta (M) : = \sqrt{ \beta^2 (M)} \geq 0$. 
\end{definition}

When $\mathfrak{C}\neq \varnothing$,  the compactness of $\hull \mathfrak{C}$  ensures that the above maximum
is actually achieved, while the fact that that $0\in \hull \mathfrak{C}$  guarantees that $\beta^2(M) \geq 0$. 
Here it is important to remember  that, since   the intersection form $\bullet$ is typically  indefinite, the function 
$\mathsf{a}\mapsto \mathsf{a}^2:= \mathsf{a} \bullet \mathsf{a}$ is usually  not   convex,  and that  the maxima of 
$\mathsf{a}^2$ on $\hull \mathfrak{C}$ therefore can (and often  do)    occur
 away from the vertices  $\mathfrak{C}$ of  the convex hull.

Using essentially the same Weitzenb\"ock arguments as before, one now deduces   the following general result \cite{lebeta}:

\begin{proposition}
\label{watch} 
 Let $M$ be a  smooth compact connected  oriented $4$-manifold
with $b_{+}\geq 2$ and $\mathfrak{C}\neq \varnothing$. Then 
\begin{equation}
\label{chai}
\mathcal{I}_s (M) \geq 32\pi^2 \beta^2 (M) \quad \mbox{and} \quad \mathscr{Y}(M) \leq - 4\pi \sqrt{2} \beta (M).
\end{equation}
\end{proposition} 

Here one might object that our soft definition of a monopole class means that we may well not be able to directly determine
the full collection $\mathfrak{C}$ of all monopole classes on a given $4$-manifold. Nonetheless, knowing any  subset of $\mathfrak{C}$ gives
a lower bound for $\beta^2$, while \eqref{chai} gives us  an upper bound for $\beta^2$ 
in terms of  $\inf \mathfrak{S} (g_j)$ for any   sequence of Riemannian metrics
$g_j$ on $M$. By combining these two facts, one can in practice  exactly calculate $\beta$ for
many interesting $4$-manifolds. For example,  using the results of Bauer and Furuta  to 
prove the existence of specfic monopole classes on suitable connected sums, and then 
constructing specific minimizing sequences $g_j$ by pasting together K\"ahler building blocks, one can prove the following  
\cite{il1,il2}:

\begin{theorem}
Let $X$, $Y$, and $Z$ be minimal simply-connected complex surfaces  with $$b_+\equiv 3\bmod 4.$$
Then, for any $k\geq 0$, 
$$\beta^2 (X\# Y\# k \overline{\CP}_2 ) = c_1^2 (X) + c_1^2 (Y)$$
and 
$$\beta^2 (X\# Y \# Z \# k \overline{\CP}_2) = c_1^2 (X) + c_1^2 (Y)+ c_1^2 (Z).$$
Moreover, equality holds in \eqref{chai} for each of these connected sums.
\end{theorem} 

For example, for any positive  integer $m$, the complex surface 
$M_{4m}\subset \CP_3$ of degree $4m$ is simply connected and has $b_+\equiv 3\bmod 4$. Choosing $X$, $Y$, and $Z$ to be 
of this form,  where $X$ is not a quartic, we thus obtain a horde of 
 concrete  simply connected $4$-manifolds  whose   Yamabe invariants are negative and {\sf explicitly calculable}.
 These examples are  significantly  different from any previously discussed here. In particular,  these  smooth manifolds 
never admit symplectic structures,  because 
  an argument first sketched by Witten \cite{witten} 
 shows that $\mathbf{SW}_\mathfrak{c} =0$ for any   spin$^c$ structure $\mathfrak{c}$ of almost-complex type 
   on a connected sum of two $4$-manifolds with $b_+\neq 0$.
 
Seiberg-Witten theory thus reveals  that an astonishing  profusion of compact simply connected topological $4$-manifolds
admit smooth structures for which the Yamabe invariant is negative.  We will now use  other  techniques 
to explore  the cases where the Yamabe invariant is 
 zero or positive.

\section{Dimension Four: the  Yamabe Zero Case}
\label{yzero}

If  $(M^4, J)$ is a compact complex
surface of {\sf K\"ahler type}, 
Theorem \ref{bacalao} says  that the sign of the Yamabe invariant  $\mathscr{Y}(M)$  is completely determined by its Kodaira dimension
$\kod (M,J)$. 
But does this pattern also hold true  when $(M^4, J)$ is of {\sf non-K\"ahler type}, or equivalently,  when $b_1(M)$ is odd? 
Since  every complex surface with $\kod (M)=2$ is of K\"ahler type, 
the  non-K\"ahler case  naturally  breaks up  into  the sub-cases  of $\kod (M)=1$, $\kod (M)=0$,
and $\kod (M)=-\infty$. However, my earlier paper \cite{lky} also settled the non-K\"ahler case when $\kod (M)=0$,
thereby only leaving 
the problem  open for  complex surfaces $(M^4,J)$ with $b_1(M)$ odd and  $\kod (M,J)= 1$ or
$\kod (M,J) =-\infty$.  Meanwhile,  Michael Albanese \cite{alba} more 
recently showed that the simple pattern linking   Kodaira dimension and the Yamabe invariant 
is sometimes violated  for complex surfaces with  $\kod =-\infty$ and  $b_1$  odd. In this section, we will 
clarify his arguments 
by putting them in a somewhat more general setting, and then show that these same ideas
nonetheless  imply  that the original pattern does  hold true for complex surfaces with   $\kod = 1$  and  $b_1$  odd. 

Any complex surface $(M,J)$ of Kodaira dimension $1$ is property elliptic \cite{bpv}, because holomorphic sections
of some some power $K^{\otimes \ell}$ of the canonical line bundle $K=\Lambda^{2,0}$ define a holomorphic map
$M\to C$ onto a smooth connected Riemann surface $C$ such that  any regular fiber is an elliptic\footnote{This inescapable piece
of  traditional terminology unfortunately clashes with most uses of the term ``elliptic'' in geometry and topology, where it is usually reserved for  spaces that are positively curved, rather than flat. The historical origin of this  peculiar usage is that elliptic curves
first arose   in  connection with the ``elliptic integral'' that expresses the  arclength of  a generic ellipse in the plane.}
 curve $E\approx T^2$. By deforming the complex structure, one can then show that such manifolds can 
always  be  obtained by starting with a $4$-orbifold that fibers over  a $2$-orbifold with flat fibers, then 
gluing in gravitation instantons,  and  finally  blowing up.  By a generalization of 
 arguments of   Cheeger-Gromov \cite{cheegro}, 
 this  allows one \cite{lky} to show that these manifolds admit  sequences of
Riemannian  metrics  with $\int s^2 d\mu \searrow 0$. It therefore follows that 
 every properly elliptic complex surface $(M^4, J)$  has 
 $\mathcal{I}_s (M) =0$, and hence $\mathscr{Y}(M)  \geq 0$. 
 To show
that these manifolds have $\mathscr{Y}(M)=0$, it therefore suffices to show that they never admit metrics of
positive scalar curvature. 

In the K\"ahler case, Seiberg-Witten theory rules out the existence of positive-scalar-curvature metrics
on properly elliptic surfaces, so their Yamabe invariants must indeed vanish. In the non-K\"ahler case,
a systematic calculation by Biquard \cite{biq2}  also revealed that many properly elliptic surfaces with $b_1$ odd 
  do still  carry  non-trivial Seiberg-Witten basic classes, thereby making it plausible  that the pattern might be independent of the 
 parity of $b_1$. Unfortunately, however, Biquard's results 
  do not suffice to settle the problem, because there do exist properly elliptic
 surfaces with $b_1$ odd that do not carry Seiberg-Witten basic classes. For example,
 let $C$ be a Riemann surface, let $L\to C$ be a complex line bundle of Chern class $1$, let 
 $L^\times \subset L$ be the complement of the zero section, and let $M= L^\times/\ZZ$, where
 the action of $\ZZ$ on $L^\times$ is generated by, say,  scalar multiplication by $2$. If $C$ has 
genus ${\zap g}\geq 2$, this elliptic surface has $\kod =1$ and $b_1= 2 {\zap g}+1$, but Biquard's 
calculation reveals that its Seiberg-Witten invariants are all zero. Finishing the job  therefore requires 
 a different method for excluding the existence
of positive-scalar-curvature metrics.

%
% We note in passing that 
%if we instead took ${\zap g}=1$, we would obtain a Kodaira surface, which has $\kod =0$, admits a symplectic structure, 
% and carries
%a basic class;  while setting ${\zap g}=0$ would give us a Hopf surface, which has $\kod =-\infty$,
%and carries metrics of positive scalar curvature.  
  
Fortunately, combining the Schoen-Yau and Gromov-Lawson methods 
yields  the  following  useful  generalization of  a theorem of   Albanese \cite{albathes,alba}: 

\begin{proposition} 
\label{expansive}
Let $N^3$ be a smooth compact oriented connected {\sf enlargeable}  $3$-manifold,
and let $X$ be a smooth compact
oriented $4$-manifold that admits a  smooth   submersion $\varpi: X\to S^1$
 with fiber $N$. Let $P$ be any smooth compact oriented $4$-manifold,
 and let $M= X\# P$. Then $\mathscr{Y}(M)\leq 0$. In other words,   $M$ does not admit metrics of
 positive scalar curvature.
\end{proposition}

\begin{proof}
Proceeding by contradiction, let us suppose that $M$ admits a metric $g$ of positive scalar curvature. 
Let ${\zap b}: X\# P\to X$ be the  smooth ``blowing down'' map that collapses $(P-B^4)\subset M$ to a point, and let 
$$f = \varpi\circ {\zap b}: M\to S^1= \mathbf{U}(1)$$
be  the induced projection. The pre-image $f^{-1} (z )$ of a regular value is thus a copy
of $N$ in $M$ whose  homology class $[N] \in H_3(M, \ZZ)\cong H^1 (M, \ZZ)$ is 
non-zero, because $M- N$ is connected.  Now  recall that 
\begin{equation}
\label{doh}
H^1(M, \ZZ) =  C^{\infty} (M, \mathbf{U}(1)) / \exp\,  [2\pi \mathsf{i}\, C^{\infty} (M, \RR)],
\end{equation}
because there is short exact sequence of sheaves of Abelian groups
$$0\to \ZZ\to C^\infty (\underline{\quad}, \RR)  \stackrel{\exp 2\pi \mathsf{i}\cdot }{\longrightarrow}   C^\infty (\underline{\quad}, \mathbf{U}(1)) \to 0,$$
where $C^\infty (\underline{\quad}, \RR)$ is a fine sheaf. We may now represent the homology class
$[N]$ by a mass-minimizing rectifiable current, and recall that, since $M$ has dimension $4< 8$,  this current is then a sum of 
smooth oriented connected hypersurface $\Sigma_1, \ldots , \Sigma_k$ with positive integer multiplicities; in particular, 
\begin{equation}
\label{play} 
[N] = \sum_{i=1}^k n_i [\Sigma_i], \qquad n_i \in \ZZ^+.
\end{equation}
By the Schoen-Yau argument recounted  in \S \ref{yammer}, the $g$-induced conformal class $[h]$ on each hypersurface
$\Sigma_i$ necessarily has positive Yamabe constant. On the other hand, 
each hypersurface $\Sigma_i$ can then be written as $f_i^{-1} (1)$ for some smooth map $f_i: M\to \mathbf{U}(1)$, 
where  $[f_i] \in H^1(M, \ZZ)$ is 
the Poincar\'e dual of $[\Sigma_i]$ and $1\in \mathbf{U}(1)$ is a regular value of $f_i$. It then follows that 
$$\widehat{f} : = \prod_i f_i^{n_i} : M\to  \mathbf{U}(1)$$
represents the Poincar\'e dual of $[\Sigma]$, where the product of course means the point-wise product
of $\mathbf{U}(1)$-valued functions. Since $f$ and $\widehat{f}$ therefore represent the same class in $H^1(M, \ZZ)$, equation 
\eqref{doh} therefore tells  that 
$$f = e^{2\pi \mathsf{i}u}  \widehat{f}$$
for some smooth real valued function $u: M\to \RR$, and we thus obtain an explicit homotopy of $\widehat{f}$ to $f$ by simply setting $f_t = e^{2\pi \mathsf{i}tu}\widehat{f}$ for $t\in [0,1]$.
Since $\widehat{f}$ is constant on each $\Sigma_i$, it  follows that $f|_{\Sigma_i}$ induces the zero homomorphism
$\pi_1 (\Sigma_i) \to \pi_1 (S^1)$, and  the inclusion map ${\zap j}_i: \Sigma_i \hookrightarrow M$ therefore 
lifts to an embedding  $\tilde{\zap j}_i:\Sigma_i\hookrightarrow \widetilde{M}$ of this hypersurface in 
the covering space $\widetilde{M}\to M$ corresponding to  the kernel of $f_* : \pi_1(M) \to \pi_1(S^1)$.

However, we can identify  $X$ with the mapping torus 
$$\taurus_\varphi := 
(N\times \RR) /\Big\langle \, (x, t) \longmapsto ( \varphi (x), t+2\pi ) \, \Big\rangle
$$
of some  diffeomorphism $\varphi: N\to N$ by simply 
 choosing a vector field on $X$ that projects to  $\partial/\partial \theta$ on $S^1$, and  then following the flow. 
Since  the diffeotype of  $\taurus_\varphi\to S^1$   moreover only 
depends on the isotopy class of $\varphi$, we may also assume that $\varphi$
has a fixed-point $p\in N$. The flow-line $\{ p\} \times \RR$ then covers an embedded circle  $S^1 \hookrightarrow X$,
which we may moreover take to avoid the ball where  surgery is to be performed to  construct $M= X\# P$. 
This circle in $X$ then also defines  an embedded circle in $M$, which we will call the {\em reference circle}; 
and by first making a small perturbation, if necessary, we may assume that this reference circle $S^1\hookrightarrow M$
is also transverse to the $\Sigma_i$. 
Since $N$ has intersection $+1$ with the reference circle, equation \eqref{doh} tells us that at least one of the hypersurfaces 
$\Sigma_i$ has non-zero intersection with the reference circle. 
Setting $\Sigma= \Sigma_i$ for some  such $i$, 
 we will  henceforth denote the corresponding inclusions map 
${\zap j}_i $ and its lift $\tilde{\zap j}_i$ by 
 ${\zap j} : \Sigma \hookrightarrow M$  and  $\tilde{\zap j} : \Sigma \hookrightarrow \widetilde{M}$.
 Our mapping torus model of $X$ now gives us a diffeomorphism $\widetilde{M}= (N\times \RR) \# ( \#_{\ell =1}^\infty P)$,
 along with  a blow-down map  $\widehat{\zap b} : \widetilde{M} \to N\times \RR$ that lifts 
 ${\zap b} : {M} \to X$. However,    $\{ p \} \times \RR\subset N \times \RR$ now meets $\widehat{\zap b}\circ \tilde{\zap j} (\Sigma )$ transversely in a set
 whose oriented count 
 exactly computes the homological intersection number $\mathfrak{n}\neq 0$
of our reference circle $S^1$ with $\Sigma$. 
Thus,  if $\mathfrak{P} : N\times \RR \to N$ denotes the first-factor projection, 
the smooth map $\mathfrak{P} \circ  \widehat{\zap b}\circ \tilde{\zap j} : \Sigma \to N$ has  degree $\mathfrak{n}\neq 0$. 
However, since $N$ is enlargeable by hypothesis, this implies \cite{gvln2} 
that  $\Sigma$ is enlargeable, too, and therefore does not admit a metric of positive scalar curvature. But 
we have also already noticed  that $\Sigma = \Sigma_i$ {must} admit a metric $h$ of positive scalar curvature
conformal to the restriction of the putative positive-scalar-curvature metric $g$ on $M$! 
This contradiction therefore  shows that $M$ cannot actually admit such  a metric $g$, and hence  that 
$\mathscr{Y}(M) \leq 0$, as claimed. 
\end{proof}

Using this, one can now prove the following satisfying result, which was discovered in the course
of conversations with Michael Albanese:

\begin{theorem}
\label{ellipsis}
Let  $M$ be the underlying smooth   $4$-manifold of
a compact complex surface $(M^4,J)$  of Kodaira dimension $0$ or  $1$. Then $\mathscr{Y}(M)=0$. 
\end{theorem}
\begin{proof}
Because the other   cases  are covered by my earlier paper\cite{lky},  we may  henceforth  assume that
$b_1(M)$ is {odd}
 and  that $\kod (M,J) =1$.  Since these previous results  also showed that $\mathscr{Y}(M) \geq 0$, we also merely need to show
 that $M$ does not admit metrics of positive scalar curvature. If $X$ is the minimal model of $(M,J)$, then, after normalization,  the pluricanonical system 
 defines a holomorphic map $X\to C$ to a smooth connected complex curve $C$.  Because $b_1$ is odd
 by assumption, there cannot \cite{brin-elliptic} be any fibers that are just  unions of rational curves. 
 Thus,  $X$ has Euler characteristic zero, and 
 $X\to C$ has  at worst  multiple fibers. Now equip $C$ with an orbifold structure by giving each point 
 a  weight equal to the multiplicity of the corresponding fiber. Because $\kod = 1$ by assumption, 
 we must have $\chi^{orb}(C) \leq 0$,  because $(M,J)$ would  otherwise \cite[\S 2.7]{FMbook} be a Hopf surface, and thus
 have $\kod = -\infty$. In particular,  $C$ must be a    {\em good orbifold}  in the sense of Thurston \cite{thurston-orb},
 so that  $C= \widehat{C}/\Gamma$, where $\widehat{C}$ is a smooth complex curve of positive genus, 
and where the finite group $\Gamma$ acts  biholomorphically on $\widehat{C}$. 
Pulling 
$X$ back to $\widehat{C}$ then produces an (unramifed)  cover $\widehat{X}\to X$ with a smooth fibration 
$\widehat{X}\to \widehat{C}$ that makes it into a bundle \cite{brin-elliptic,wall-elliptic} of elliptic curves. Since $\kod (\widehat{X}) = 1$,
the base $\widehat{C}$ must have genus $\geq 2$, and, 
 at the price
of perhaps replacing $\widehat{C}$ with an unbranched cover, we can then kill the rotational monodromy \cite{bpv},
and thereby  arrange for   $\widehat{X}\to \widehat{C}$
to just  be a principal $[\mathbf{U}(1)\times \mathbf{U}(1)]$-bundle over $\widehat{C}$. Moreover, 
this cover $\widehat{X}$ must still have 
$b_1$ odd, because if $\widehat{X}$ admitted a K\"ahler metric,
the fiberwise-average of its local push-forwards would then produce a forbidden K\"ahler metric on $X$. 
The Chern classes of the two $\mathbf{U}(1)$-factors must therefore be linearly dependent over $\mathbb{Q}$,
and by again passing to a cover  and then changing basis if necessary, one may arrange for  
exactly one of these Chern classes to be  non-zero. This reduces the problem to 
the case of  $\widehat{X} \approx N\times S^1$, where  $N\to \widehat{C}$ is a circle bundle of non-zero degree 
over a Riemann surface of genus $\geq 2$. However, this $N^3$ is enlargeable, because \cite{gvln2}
it contains a homologically non-trivial  $2$-torus, namely the preimage of any homologically non-trivial $S^1\subset \widehat{C}$. 
Proposition \ref{expansive} therefore says that $\widehat{X}\# k \overline{\CP}_2$ cannot admit positive-scalar-curvature
metrics for any $k$. But since $M$  has a covering space that is precisely of this form, it  follows that 
$M$ cannot admit positive-scalar-curvature
metrics either. Since we also know \cite{lky} that $\mathscr{Y}(M) \geq 0$, it therefore  follows that $\mathscr{Y}(M)=0$, as claimed. 
 \end{proof}

This   immediately implies Theorem \ref{parity}, and so might    give one the false impression 
that  the parity of $b_1$ should have little effect on results like Theorem 
\ref{bacalao}. However, life becomes significantly more complicated when $\kod = -\infty$. When $b_1$ is even, the surfaces of 
Kodaira dimension $-\infty$ are all {\em rational} or {\em ruled}, meaning that they are exactly 
  $\CP_2$,  $\CP_1$-bundles over  Riemann surfaces, and their blow-ups; 
and Theorem 
\ref{bacalao} is  made possible, in part,  by the fact that every manifold on this list 
admits Riemannian metrics of positive scalar curvature. 
By contrast, no complete classification is  currently  available for  the complex surfaces with 
$b_1$ odd and Kodaira dimension $-\infty$, which  are known, for historical reasons, as 
{\sf surfaces of class  {VII}}.  
The most familiar  class-{\sf VII} surfaces  are the (primary) Hopf surfaces $(\CC^2 - \{ 0\})/\ZZ$, which are diffeomorphic to 
$S^3\times S^1$; and since  these and their blow-ups $\approx (S^3\times S^1) \# k\overline{\CP}_2$ obviously
also admit metrics of positive scalar curvature, one might hope for the same pattern to continue  to hold even 
 when $b_1$ is odd. However, Inoue \cite{inoue} discovered families  of minimal 
class-{\sf {VII}}  surfaces that are  topologically  quite  unlike Hopf sufaces,   and  
Michael Albanese \cite{albathes,alba} recently showed that this topological difference has a major impact 
on the Yamabe invariant:

\begin{theorem}[Albanese]
\label{michael} 
Let $X$ be an Inoue surface, in the above sense, and let $M$ be obtained from 
$X$ by blowing up $k\geq 0$ points. Then $\mathscr{Y}(M)=0$. 
\end{theorem}
\begin{proof} Inoue's examples, which were all constructed as   quotients of a half-space in $\CC^2$, 
are grouped into three families, but each such  complex surface $X$  is 
diffeomorphic \cite{alba,inoue} to the mapping torus $\taurus_\varphi$ of some self-diffeomorphism $\varphi$
of  a $3$-manifold  $N$. For one family, $N$ is  the  $3$-torus $\mathbb{T}^3$,
and $\varphi\in SL(3, \ZZ)$ is an affine map, while  for the others, 
$N\to \mathbb{T}^2$ is a 
non-trivial  circle bundle  over the $2$-torus, and $\varphi$ is an automorphism
of its nilgeometry. 
Since any such $N$ is enlargeable, and since any blow-up $M$ is diffeomorphic to 
$X\# k \overline{\CP}_2$ for some $k\geq 0$, Proposition \ref{expansive} therefore tells us   that $\mathscr{Y}(M)\leq 0$. 
On the other hand, each such $X$ admits an $F$-structure of positive rank, in the sense of 
Cheeger-Gromov \cite{cheegro},  and so admits families  of metrics which  collapse
to zero volume with bounded curvature. Modifying such metrics by inserting scaled-down 
Burns metrics on $\overline{\CP}_2 - \{ pt\}$,
we thus obtain sequences of metrics on any blow-up $M$ with $\mathfrak{S}\searrow 0$. 
This shows that 
 $\mathcal{I}_s(M)=0$, and hence $\mathscr{Y}(M) \geq 0$. Putting these two observations together,
we therefore have $\mathscr{Y}(M) =0$, as claimed. 
\end{proof}

Theorem \ref{michael} at least  implies that all  known class-{\sf VII} surfaces therefore   have non-negative Yamabe invariant, 
and  the  global-spherical-shell conjecture  credibly claims  that no other  class-{\sf VII} surfaces
remain to be discovered. 
See Dloussky and Teleman\cite{dloutele} for 
an  overview of    this classification problem. 

Finally, while Theorem \ref{ellipsis} focused on the notion of
enlargeability because it behaves well under maps of non-zero degree, it is perhaps also worth noting that this 
hypothesis is in fact {\sf optimal}. Indeed, 
 Ricci-flow methods show \cite{bbbpm,lott} that a
smooth compact oriented $3$-manifold $N$ admits a positive-scalar-curvature metric if and only if 
it is a connected sum of spherical space forms $S^3/\Gamma_j$ and/or  copies of $S^2 \times S^1$.
Geometrization goes on to tell us that any
 other $3$-fold either contains  an incompressible $2$-torus or  has an enlargeable connect-summand that 
is a compact $K(\pi_1)$ modeled on one of the  other six Thurston geometries \cite{thur3}. Because 
any $3$-manifold that 
contains an incompressible torus  is also enlargeable, and because 
the connect sum of
an enlargeble manifold with any spin manifold is also enlargeable \cite{gvln2}, it therefore follows that a $3$-manifold $N$ has 
 $Y(N) \leq 0$ iff it is enlargeable. Nonetheless, when  
  $Y(N^3) \leq 0$, the same reasoning  also says that one can 
 detect  this fact using the   results of Schoen-Yau \cite{sy3man}, because
  Agol's proof \cite{Agol} of the virtual Haken conjecture implies that
  any such $N$  has a finite cover that, for any metric, must  contain  a stable minimal 
 surface of genus $\geq 1$.

\section{Dimension Four: the  Yamabe Positive Case}
\label{ypos} 

If  $(M^4,g,J)$ is  a compact K\"ahler-Einstein manifold with negative scalar curvature,
Theorem \ref{optimal} asserts that $g$ actually realizes $\mathscr{Y}(M)$. The same conclusion  also holds 
in the Ricci-flat case, because such an  $M$ is   
finitely covered by either  by  $K3$ or  $\mathbb{T}^4$, and so cannot admit positive-scalar-curvature metrics. 
 Unfortunately, however,  life becomes significantly more complicated in the positive case. 
  Nonetheless, the same phenomenon does still occur \cite{gl1,lcp2} in one interesting case  here: 

\begin{theorem} 
\label{topper} 
The Fubini-Study metric on $\CP_2$ realizes  the Yamabe invariant $\mathscr{Y}(\CP_2)$. In other words, 
$$\mathscr{Y}(\CP_2 ) = 4\pi \sqrt{2 c_1^2(\CP_2)} = 12 \pi \sqrt{2}.$$
Moreover, modulo diffeomorphisms and  rescalings, the Fubini-Study metric is the unique 
Yamabe metric that achieves this  minimax. 
\end{theorem}

While this was originally discovered \cite{lcp2} using the perturbed Seiberg-Witten equations,
we will instead describe the simpler proof later given in my joint paper with Gursky\cite{gl1}. Observe that, for the 
spin$^c$ structure on $\CP_2$ induced by the usual complex structure,
the spin$^c$ Dirac operator 
$$\dir_\theta : \Gamma (\mathbb{V}_+) \to \Gamma (\mathbb{V}_-)$$
has index $1$, because this is the Todd genus $h^{0,0}-h^{0,1}+h^{0,2}$ of $\CP_2$. Now, for an arbitary
metric $g$ on $\CP_2$, choose the connection $\theta$ so that $\frac{i}{2\pi} F_\theta$
is the unique harmonic $2$-form  representing $c_1(L) = c_1(\CP_2, J)$. Because 
$\CP_2$ has positive-definite intersection form $\bullet$, this harmonic form  is therefore self-dual, 
and  we therefore have 
\begin{equation}
\label{sizemup}
\|  F_\theta\|^2_{L^2} = 4\pi^2 c_1^2(\CP_2) = 36\pi^2.
\end{equation}
However, since $\Ind \dir_\theta > 0$, the twisted spin bundle $\mathbb{V}_+$  must have a smooth section 
$\Phi\not \equiv 0$  with $\dir_\theta\Phi =0$. Applying the Weitzenb\"ock formula \eqref{wtw}, we therefore
have
\begin{eqnarray*}
0 &=& 2\Delta |\Phi |^2 + 4|\nabla_{\theta} \Phi |^2 + 
s |\Phi |^2 + 8 \langle -iF_{\theta}^{+} , \sigma (\Phi ) \rangle\\
&\geq & 2\Delta |\Phi |^2 + (s - 2\sqrt{2} |F_{\theta}|) |\Phi|^2 
\end{eqnarray*}
so that 
\begin{equation}
\label{cool} 
0 \geq \int_M (s_g - 2\sqrt{2} |F_{\theta}|_g) |\Phi|^2d\mu ~,
\end{equation}
with equality only if $\Phi$ is parallel. However, this would then force $\sigma (\Phi )$ to consequently
be a non-zero parallel section of $\Lambda^+$, 
so equality in \eqref{cool} can only happen if $g$ is K\"ahler. 

Let's now see what this tells us about conformal rescalings $\widehat{g} = u^2 g$ of our metric $g$. 
One  key observation   is that {\sl harmonic 2-forms are conformally invariant in dimension four}, 
so our  prescription for $F_\theta$ is  unchanged by a conformal change of metric. 
On the other hand, the point-wise norm of $F_\theta$  is of course  {\em not} conformally invariant; 
instead, its norms with respect to $g$ and  $\widehat{g} = u^2 g$ are related by 
$$|F_\theta|_{\widehat{g}} = u^{-2} |F_\theta|_{g}.$$
Thus, the Yamabe equation 
$$
\widehat{s} = u^{-3}(6\Delta + s) u
$$  
implies that the ``perturbed scalar curvature''
\begin{equation}
\label{kazaam}
\mathfrak{s} : = s - 2\sqrt{2} |F_{\theta}|
\end{equation}
satisfies a perfect analog
$$\widehat{\mathfrak{s}} = u^{-3}(6\Delta + \mathfrak{s}) u$$
of the Yamabe equation. (Here,  of course,  $\mathfrak{s}_g$ has been abbreviated as $\mathfrak{s}$,
while $\widehat{\mathfrak{s}}$ is used as short-hand for $\mathfrak{s}_{\widehat{g}}$.) However, 
if we now let 
$\lambda_{\mathfrak{s}}$ denote the smallest eigenvalue of $6\Delta + \mathfrak{s}$,  then the Rayleigh-quotient characterization of 
$\lambda_\mathfrak{s}$ implies that corresponding eigenspace
is $1$-dimensional, and is spanned by an everywhere-positive function \cite[Theorem 8.38]{giltrud}. Thus,  there
is a  $C^2$ function $u>0$ on  $M$ with 
$$(6\Delta + \mathfrak{s}) u= \lambda_{\mathfrak{s}} u ,$$
and the corresponding rescaled metric $\widehat{g} = u^2 g$ therefore satisfies  $\widehat{\mathfrak{s}}= \lambda_{\mathfrak{s}}u^{-2}$
everywhere. 
In particular, this means that $\widehat{\mathfrak{s}}$ has the same sign at every point! 
But \eqref{cool}, applied to $\widehat{g}$, tells us that this rescaled metric must also satisfy
$$0\geq \int_M \widehat{\mathfrak{s}} \,|\widehat{\Phi}|^2 d\mu_{\widehat{g}}$$
for some section $\widehat{\Phi}\not\equiv 0$ of $\mathbb{V}_+$, with  equality only if $\widehat{g}$ is  K\"ahler. 
Thus, either $\widehat{\mathfrak{s}} < 0$ everwhere, or else $\widehat{\mathfrak{s}} \equiv  0$ and $\widehat{g}$ is K\"ahler. 
In either case, \eqref{kazaam} and the Cauchy-Schwarz inequality  now tell us that 
$$  
\int_M \widehat{s}\, d\mu_{\widehat{g}} \leq 2\sqrt{2} \int_M   |F_{\theta}|_{\widehat{g}} \, d\mu_{\widehat{g}} \leq 2\sqrt{2} 
\left(\int_M |F_{\theta}|^2_{\widehat{g}} \, d\mu_{\widehat{g}}\right)^{1/2} \left( \int_M 1~ d\mu_{\widehat{g}}\right)^{1/2} 
$$
so that 
$$
\mathscr{E} (\widehat{g}) = \frac{\int_M {s}_{\widehat{g}}\, d\mu_{\widehat{g}}}{\sqrt{\Vol (M,{\widehat{g}})}} 
\leq  2\sqrt{2} \| F_{\theta}\|_{L^2}  = 12\pi\sqrt{2}
$$
by \eqref{sizemup}. 
Minimizing $\mathscr{E}$ over the conformal class $[g]$, we consequently have 
\begin{equation}
\label{punch}
{Y} (\CP_2 , [g]) \leq 12\pi\sqrt{2}. 
\end{equation}
Since  any Einstein metric is a Yamabe metric by Obata's theorem, and since 
the Fubini-Study metric is an Einstein metric with $\mathscr{E}= 12\pi\sqrt{2}$, the conformal class of the Fubini-Study 
metric exactly saturates the upper bound \eqref{punch}, and we therefore have   $\mathscr{Y}(\CP_2)= 12\pi\sqrt{2}$, 
as claimed.

If equality  held in  \eqref{punch}, the  constructed conformal metric 
$\widehat{g}$ would have to be   both  a Yamabe minimizer and a K\"ahler metric. 
In particular,  $\widehat{g}$  would necessarily    be a K\"ahler metric of constant positive scalar curvature. 
But since any  constant-scalar-curvature K\"ahler metric has harmonic  Ricci form, and since 
$b_2(\CP_2)=1$, this would force the Ricci form to be a constant times the K\"ahler form, thus making 
$\widehat{g}$ a K\"ahler-Einstein metric of 
positive Einstein constant. In particular, Obata's theorem would then say that $\widehat{g}$
is the {\em unique} Yamabe minimizer in the given conformal class $[g]$. 
Let us now normalize $\widehat{g}$ by a constant rescaling, and thereby give it Ricci curvature $6$ and scalar curvature $24$.
The assumption of equality in 
\eqref{punch} then says that 
the  volume of $(\CP_2, \widehat{g})$ is $\pi^2/2$. However, if we now give the unit  circle bundle $S\to \CP_1$ in 
$ K^{1/3}\cong \mathcal{O}(-1)$  the Riemannian  submersion metric ${\zap h}$ 
for which the standard vertical $\partial/\partial \theta$ vector field is assigned length $1$, 
this construction then yields  \cite[Theorem 9.76]{bes} an Einstein metric on $S\approx S^5$; indeed,  this is  the standard 
 Sasaki-Einstein metric \cite{bg} associated with the normalized K\"ahler-Einstein metric $\widehat{g}$. 
 However, since $(S,{\zap h})$ then has same volume $2\pi (\pi^2/2) = \pi^3$  and
the same Ricci curvature $4$ as the standard metric on $S^5$, it saturates the  Bishop-Gromov inequality \cite{bes,bishop},
and must therefore be isometric to the standard unit $5$-sphere.  Since the $S^1$ orbits  of the Killing field 
$\partial/\partial \theta$ moreover all have the same length,
the group of isometries it generates  must just  be  a diagonally embedded  
$S^1$ in the  maximal torus
$S^1 \times S^1\times S^1$ of $\mathbf{SO}(6)$. The Riemannian submersion $S^5\to S^5/S^1$
is therefore standard, and the submersion  metric on the base 
$S^5/S^1$ 
must  
therefore be isometric to the  usual Fubini-Study metric  on $\CP_2$.  This  shows
  that   equality  holds in    \eqref{punch}  if and only if 
  $[g]$ is    the conformal class of   the Fubini-Study metric, albeit typically pulled back by some self-diffeomorphism of $\CP_2$. 

\bigskip

It is not difficult to generalize this  to other $4$-manifolds with strictly positive   intersection form $\bullet$. Indeed, 
essentially the same argument proves the following \cite{gl1}:

\begin{theorem} 
\label{definitely} 
Let $k\in \{1,2,3\}$, and let 
$\ell$ be any natural number. Then 
$$12\pi \sqrt{2} \leq 
\mathscr{Y}(k{\CP}_2\#  \ell [S^1\times S^3]) \leq 4\pi \sqrt{2k+16}.$$
In particular, these connected sums of copies of ${\CP}_2$ and
$S^1\times S^3$  all 
have Yamabe invariant strictly less than 
$\mathscr{Y}(S^4)=8\pi \sqrt{6}$.
\end{theorem} 

Here the upper bound is proved by applying the previous argument to a spin$^c$ structure
with  $c_1(L)= (3)$, $(3,1)$, or $(3,1,1)$ relative to a basis for  $H^2(\ZZ) \cong \ZZ^k$ in which
the intersection form $\bullet$  is diagonalized. 
The corresponding spin$^c$ Dirac operator then  again has 
index $1$, 
 so the proof is essentially the same as before, except that  $c_1^2(L) = 8+k$, and that,  when $\ell \geq 1$ or $k\geq 2$, 
equality can  immediately  be ruled out in \eqref{cool}
because these  Yamabe-positive $4$-manifolds are never even   homotopy-equivalent to rational or ruled surfaces, and so
cannot admit K\"ahler metrics by 
Theorem \ref{bacalao}. 
 On the other hand, the lower bound for the Yamabe constant follows from the Kobayashi inequality 
\eqref{kobble}, together with the fact,   discovered independently by Kobayashi \cite{okob} and Schoen \cite{sch},
that $\mathscr{Y}(S^1\times S^{3}) = \mathscr{Y}(S^4)$. 

\medskip 

One surprising consequence  of Theorem \ref{definitely} is the following \cite{gl1}: 

\begin{corollary}
\label{really}
 The $\kod \geq 0$  hypothesis in Theorem \ref{parity} is essential, because 
a primary  Hopf surface $\approx S^1\times S^3$ and its one-point blow-up $\approx [S^1 \times S^3] \# \overline{\CP}_2$
have different Yamabe invariants. 
\end{corollary} 
Indeed, since  the Yamabe invariant is manifestly independent of orientation,  Theorem \ref{definitely} 
 tells us  that 
$$\mathscr{Y}([S^1 \times S^3] \# \overline{\CP}_2) = 12\sqrt{2}\pi < 8\sqrt{6}\pi = \mathscr{Y}(S^1\times S^3).$$
Of course, this example involves complex surfaces with $b_1$ odd, and so says nothing about
the case where  $b_1$ is even. One peculiarity  is that,  while we could have just as well  used 
the blow-up of a Hopf surface at $2$ or $3$ points as the key exhibit in Corollary \ref{really}, 
the argument would fail if we instead blew up $4$ or more
points.  It is an  open problem to determine whether this limitation of Theorem \ref{definitely} 
is    a  mere technical artifact, or whether it   genuinely  reflects of some    undiscovered geometric 
phenomenon. Settling  this question either way might easily  shed new light
on  other  open problems in the subject.

 For complex surfaces 
 with $b_1$ even and $\kod = -\infty$, the behavior of the Yamabe invariant under blowing up 
 still  remains to be  elucidated, but could turn out to be  rather more satisfying. These complex surfaces are exactly the 
 rational or ruled ones, and their underlying smooth compact  oriented $4$-manifolds are 
 characterized \cite{liu1,ohno} by the positivity of the  Yamabe invariant and the existence of an orientation-compatible symplectic
 structure. Among these, the non-spin, simply connected   ones are just  the $4$-manifolds
 $\CP_2 \# k \overline{\CP}_2$   arising from  blow-ups of $\CP_2$. For these,  the Kobayashi
 inequality \eqref{kobble} 
 immediately tells us that
 $$\mathscr{Y} (\CP_2 \# k \overline{\CP}_2) \geq \mathscr{Y} (\CP_2) = 12\pi \sqrt{2},$$
 so blowing up $\CP_2$ certainly never  decreases the Yamabe invariant; rather,  the open question is 
 whether blowing up could {\em increase} it in this context. While  the method used to 
 prove Theorem  \ref{topper} does not allow us to settle this,  it  nonetheless does provide some interesting partial information. 
 
 Indeed, if $(M,g)$ is a smooth compact oriented Riemannian $4$-manifold with indefinite intersection form $\bullet$,  
 we saw in \eqref{deco-harm} that 
 $$H^2(M,\RR) = \mathcal{H}^+_g \oplus \mathcal{H}^+_g , $$
 where $\mathcal{H}^\pm_g$ consists of cohomology classes whose harmonic representative with respect to $g$ is 
 self-dual (respectively, anti-self-dual). In particular, for any metric $g$ and any spin$^c$ structure $\mathfrak{c}$ on $M$, 
 we thus obtain a cohomology class $c_1^+ = [c_1(L)]^+\in \mathcal{H}^+_g$ as the orthogonal projection of $c_1= c_1(L)$ with respect to 
 $\bullet$. We can then repeat the argument previous by again taking $i F_\theta$ to be the harmonic 
 representative of $2\pi c_1$, but now carefully noting  that $iF_\theta^+$ is now the harmonic representative
 of $2\pi c_1^+$, and that only $F^+_\theta$ appears in \eqref{wtw}. 
 Assuming that the index
 $$\Ind \dir_\theta = \frac{c_1^2(L) - \tau (M)}{8}$$
 is positive, much the same argument then shows that any conformal class $[g]$ on $M$ then satisfies \cite{gl1} 
 \begin{equation}
\label{faible} 
 Y(M, [g]) \leq 4\pi \sqrt{2 (c_1^+)^2}.
\end{equation}
This should have a familiar ring to it, because we previously saw $(c_1^+)^2$ arise in a Seiberg-Witten 
estimate \eqref{amazing} that we used to bound the Yamabe invariant from above. But in the Seiberg-Witten setting, 
we got $-4\pi \sqrt{2 (c_1^+)^2}$ as an upper bound for $Y(M, [g]) $, whereas in \eqref{faible}   this expression 
instead occurs with the opposite sign. However, while 
$$
(c_1^+)^2\geq c_1^2(L)
$$
for any metric $g$, 
this is much less useful in the context of  \eqref{faible} than in the setting of  \eqref{amazing}, because 
$(c_1^+)^2$ will  become larger and larger as $\mathcal{H}^+_g$ tilts further and further   away from $c_1=c_1(L)$,
so that, when $b_-\neq 0$, the upper bound \eqref{faible} will  actually be weaker that the Aubin bound \eqref{aubineq}
for many  conformal classes $[g]$. Of course, one can partially remedy this by considering many different
spin$^c$ structures on $M$, but this does not seem to ever lead to definitive results on the Yamabe invariant. 

Nonetheless, \eqref{faible}  does still have  interesting ramifications for  Yamabe invariants of specific conformal classes. 
For example, on a smooth compact oriented $4$-manifold $M$ with $b_+=1$,  one can fix some specific $[\omega ]\in H^2(M, \RR)$
with $[\omega]^2 > 0$, and  only consider those metrics $g$ for which $\mathcal{H}^+_g = \spanner ~ [\omega ]$. For any
such metric and any spin$^c$ structure $\mathfrak{c}$ for which the Dirac operator has positive index, \eqref{faible} then 
takes the more concrete form 
$$
 Y(M, [g]) \leq 4\pi \,  \frac{|c_1(L) \bullet [\omega]|}{\sqrt{[\omega]^2/2}},
 $$
with equality iff  some  Yamabe minimizer  ${g}^\prime\in [g]$ 
 is a K\"ahler  metric of non-negative-constant scalar curvature that is  compatible with a complex structure $J$ 
for which $c_1(M,J)= c_1(L)$.  In particular, if $g$ is a metric such that $c_1(L) \in \mathcal{H}^+_g$, then 
$$
 Y(M, [g]) \leq 4\pi \sqrt{2 c_1^2(L)} ,$$
 with equality iff the (unique) Yamabe minimizer of $[g]$ is K\"ahler-Einstein. This puts Theorem \ref{topper} in a broader context,
 and should also be compared with Theorems \ref{optimal} and \ref{definitely}.

 Nonetheless, when the intersection form $\bullet$ is indefinite,
  K\"ahler-Einstein metrics with positive Einstein constant never achieve the Yamabe invariant; that is, they
   are never  {\sf supreme}\cite{lebyam} Einstein metrics. Indeed, the non-spin
  $4$-manifolds that carry such metrics \cite{ty,sunspot} are exactly $\CP_2 \# k \overline{\CP}_2$, $k= 3, \ldots , 8$, and
 on any of these the Einstein-Hilbert action of a  K\"ahler-Einstein metric has $\mathscr{E} = 4\pi \sqrt{2 (9-k)} < \mathscr{Y}(\CP_2)$, 
 so the Kobayashi inequality \eqref{kobble} says that one can do better via a sequence that tends toward the Fubini-Study conformal class. 
 Similarly, while  $\CP_2 \# k \overline{\CP}_2$ admits an Einstein metric for $k=1,2$ that is conformally K\"ahler (but not K\"ahler),
 these Einstein metrics have $\mathscr{E} < 4\pi \sqrt{2 (9-k)}$, and so, again, do not achieve  the Yamabe invariant. 
 Indeed, \eqref{faible} shows that any maximizing sequence must involve metrics for which $\mathcal{H}^+\cong \RR$ becomes  
  as far away
 from  $c_1(J)$ as the pull-back $H^2(\CP_2, \RR) \hookrightarrow H^2(\CP_2\# k \overline{\CP}_2, \RR)$. 
 Moreover, since  \cite{gl2,G1} any oriented non-symmetric $4$-dimensional Einstein 
 manifold $(M,g)$  must  have 
  Einstein-Hilbert action $\mathscr{E}(g)\leq 4\pi \sqrt{2(2\chi + 3\tau)(M)}$, the existence of a supreme Einstein metric on 
  $\CP_2 \# k \overline{\CP}_2$ is  {\sl a priori} impossible for any  $k \neq 0$.  
   
The classification of del Pezzo surfaces \cite{delpezzo} 
tells us that there is only one other  $4$-manifold that can admit a  positive-scalar-curvature K\"ahler-Einstein metric, namely 
the spin manifold $S^2 \times S^2= \CP_1\times \CP_1$. Here, the 
Kobayashi inequality \eqref{kobble} tells us nothing at all about the Yamabe invariant, 
but a closer examination of this example allows us to draw similar, albeit weaker, conclusions. Indeed, given a positive real number
$t$,  consider the homogeneous K\"ahler metric $g_t$ on $S^2\times S^2$ given by the 
Riemannian product of two standard  metrics on the $2$-sphere $S^2\subset \RR^3$, one of Euclidean radius $t^{1/2}$ and the other of Euclidean radius $t^{-1/2}$.
This K\"ahler metric is then Einstein iff $t=1$. Our choice of radii  guarantees that  the volume of 
$(S^2\times S^2, g_t)$ is always $16\pi^2$, but that its scalar curvature is $2( t+ t^{-1})\geq 4$. Thus, the Einstein-Hilbert action of 
$(S^2\times S^2, g_t)$  is $\mathscr{E}=8\pi( t+ t^{-1})$, which has a unique minimum of $16\pi$ at $t=1$, 
corresponding to  the K\"ahler-Einstein metric. A beautiful argument of B\"ohm, Wang, and Ziller  \cite[Theorem 5.1]{bwz}
now shows, for $t$ in some interval $(1-\epsilon, 1+\epsilon)$, the metrics $g_t$ must be the unique Yamabe minimizers of volume $16\pi^2$ in 
their conformal classes. Indeed, if this were not the case, we would be able to choose a sequence of   Yamabe minimizers
$\widehat{g}_{t_j} \in [g_{t_j}]$, $\widehat{g}_{t_j} \neq {g}_{t_j} $, of identical  volume, with $t_j\to 1$. A compactness argument then shows that these would
necessarily converge to a Yamabe metric in $[g_1]$, and since Obata's theorem tells us that that $g_1\in [g_1]$ is the unique
Yamabe minimizer of  the specified volume, we therefore have $\widehat{g}_{t_j} \to g_1$. However, an inverse-function-theorem 
argument shows that, for $t$ in some interval $(1-\epsilon, 1+\epsilon)$,  the $g_t$ are, up to constant scale, 
 the only constant-scalar-curvature metrics  in their conformal classes that are close to  $g_1$ in the $C^{2,\alpha}$ topology. 
Thus, the $g_t$ must  
 be (unique) Yamabe minimizers when $t$ is sufficiently close to $1$. In particular,  the Einstein metric
$g_1$ is not supreme, and $\mathscr{Y}(S^2 \times S^2) > 16\pi$. 

This makes it irresistible  to ask precisely which  of the homogeneous metrics $g_t$ are actually Yamabe minimizers. 
However, separation of variables reveals that  the first positive eigenvalue of the Laplacian $\Delta$ of $g_t$ 
is given by $\lambda_1= \min (2t, 2/t)$. This now allows us to show that some  $g_t$ are {\em not} Yamabe minimizers. Indeed, let 
$f$ be a $\lambda_1$-eigenfunction  of unit $L^2$-norm, and let us now compute the Yamabe 
energy of $\widehat{g}= u^2 g$, where  $u=1+\varepsilon f$ for   $\varepsilon$ small. 
We  obtain   
$$
\mathscr{E}((1+\varepsilon f)^2g_t)= \frac{\int [6\varepsilon^2 |\nabla f|^2+ s (1+ \varepsilon f)^2]d\mu}{\sqrt{\int (1+ \varepsilon f)^4 d\mu}}
= sV^{1/2} + 6(\lambda_1 - \frac{s}{3}) \varepsilon^2 V^{-1/2} + O (\varepsilon^3)
$$
where $V=16\pi^2$ is the volume of $g_t$. 
Thus, if $g_t$ is a Yamabe metric, we must have $\lambda_1 \geq \frac{s}{3}$, and our computations of 
$s$ and $\lambda_1$ therefore reveal that  a necessary condition for $g_t$ to be a
Yamabe metric is that  $t\in [\frac{1}{\sqrt{2}}, \sqrt{2}]$. 
Amazingly  enough, however, we also have 
$$\mathscr{E} (g_t) = 8\pi (t+ t^{-1}) \leq 12\pi\sqrt{2} = \mathscr{Y} (\CP_2) \quad \Longleftrightarrow \quad t\in  [\frac{1}{\sqrt{2}}, \sqrt{2}].$$
This suggests a continuity argument to show that $\mathscr{Y}(S^2 \times S^2) \geq \mathscr{Y} (\CP_2)$. Indeed, the set of
$t\in [1, \sqrt{2}]$ for which $g_t$ is a Yamabe minimizer is {\em a priori}  {\sf closed}, because the Yamabe constant is a continuous  function on the space of metrics.  
On the other hand, a straightforward  generalization of the B\"ohm-Wang-Ziller argument shows that the set of $t\in [1, \sqrt{2})$ for which $g_t$ is 
the {\sf unique} Yamabe minimizer is  {\em a priori}  {\sf open}, because the inverse-function-theorem part of the argument 
works as long as $\lambda_1 > \frac{s}{3}$. This reduces the problem to showing that this smaller set is also closed, which sounds
plausible, but  could  be quite difficult.

Finally,  Petean and Ruiz \cite{petean-ruiz} have shown that $\mathscr{Y}(S^2 \times \Sigma_{\zap g}) > \frac{2}{3} \mathscr{Y}(S^4)$, where  
$\Sigma_{\zap g}$ is a Riemann surface of any genus ${\zap g}$;  and the same proof moreover appears to also work for  twisted products. 
In conjunction with \eqref{kobble}, this implies that  the Yamabe invariants of all rational or ruled surfaces actually lie in a comparatively
narrow range.  However, determining   the precise value of the Yamabe invariants for  these manifolds continues to represent the sort of  daunting  challenge that illustrates just how much of this territory still remains  mysterious  and   largely unexplored.

\section*{Acknowledgements} The author would like to thank Michael Albanese  
for many useful conversations. He would also like to thank 
Ian Agol,  Misha Gromov, and Jimmy Petean  for useful comments and suggestions. 
This research was supported in part by a grant (DMS-1906267) from the 
National Science Foundation.

\pagebreak

\printindex 
  
\end{document}